\documentclass[11pt,a4paper]{article}
\usepackage[T1]{fontenc}
\usepackage[margin=25mm]{geometry}
\usepackage{lmodern}
\usepackage{xcolor}
\usepackage{amsmath,amssymb,amsthm,mathtools,booktabs,microtype,graphicx}
\usepackage[hidelinks]{hyperref}
\hypersetup{pdftitle={Fractional Wiener chaos: Part 2},pdfauthor={Elena Boguslavskaya and Elina Shishkina}}
\mathtoolsset{showonlyrefs}
\newcommand{\R}{\mathbb R}
\newcommand{\C}{\mathbb C}
\newcommand{\N}{\mathbb N}
\newcommand{\E}{\mathbb E}
\newcommand{\dd}{\,\mathrm d}
\newcommand{\He}{\operatorname{He}}
\newcommand{\Dom}{\operatorname{Dom}}

\newcommand{\Span}{\operatorname{span}}
\newcommand{\ind}{\mathbf 1}
\providecommand{\doi}[1]{\href{https://doi.org/#1}{\nolinkurl{https://doi.org/#1}}}
\newcommand{\subclassname}{\textbf{Mathematics Subject Classification (2020)}\enspace}

\theoremstyle{plain}
\newtheorem{theorem}{Theorem}[section]
\newtheorem{proposition}[theorem]{Proposition}
\theoremstyle{definition}
\newtheorem{definition}[theorem]{Definition}

\makeatletter
\newcommand{\subtitle}[1]{\gdef\FWC@subtitle{#1}}
\newcommand{\titlerunning}[1]{} 
\newcommand{\institute}[1]{\FWC@splitinstitutes#1\FWC@stop}
\long\def\FWC@splitinstitutes#1\and#2\FWC@stop{%
  \gdef\FWC@instituteA{#1}%
  \gdef\FWC@instituteB{#2}%
}
\DeclareRobustCommand{\at}{\\}
\newcommand{\email}[1]{\textit{E-mail:}\enspace\href{mailto:#1}{\nolinkurl{#1}}}
\def\FWC@subtitle{}
\let\FWC@originalmaketitle\maketitle
\def\FWC@formatauthors#1\and#2\FWC@stop{%
  \def\@author{#1\thanks{\FWC@instituteA}\and
    #2\thanks{\FWC@instituteB}}%
}
\renewcommand{\maketitle}{%
  \g@addto@macro\@title{\\[0.35em]{\large\FWC@subtitle}}%
  \expandafter\FWC@formatauthors\@author\FWC@stop
  \FWC@originalmaketitle
}
\makeatother

\newcommand{\keywords}[1]{%
  \par\medskip\noindent\textbf{Keywords}\enspace
  \begingroup\def\and{\unskip\enspace\textperiodcentered\enspace}#1\endgroup\par
}
\newcommand{\subclass}[1]{%
  \par\smallskip\noindent\subclassname
  \begingroup\def\and{\unskip\enspace\textperiodcentered\enspace}#1\endgroup\par
}

\begin{document}
\title{Fractional Wiener Chaos: Part 2\\ Interface Spectral Chaos}
\titlerunning{Fractional Wiener chaos: Part 2}
\author{Elena Boguslavskaya \and Elina Shishkina}
\institute{E. Boguslavskaya (corresponding author)\at
Department of Mathematics, Brunel University, Uxbridge UB8 3PH, UK\\
\email{elena@boguslavsky.net}
\and E. Shishkina\at
Department of Mathematical and Applied Analysis, Voronezh State University, Voronezh, Russia;\\
Department of Applied Mathematics and Computer Modeling, Belgorod State National Research University, Belgorod, Russia;\\
Institute of Mathematics, Physics and Information Technology, Kadyrov Chechen State University, Grozny, Russia;\\
International Laboratory of Stochastic Analysis and its Applications, National Research University Higher School of Economics, Moscow, Russia}
\date{}
\maketitle
\raggedbottom
\begin{abstract}
We construct a product spectral decomposition from power-normalised parabolic-cylinder functions. At noninteger real orders, these functions are not square integrable with respect to the corresponding full-line Gaussian measure. Bringing the admissible half-line branches into the same Gaussian space and matching them at the origin gives a corrected orthonormal eigenbasis. A unitary transformation to a weighted probability space makes the eigenfunction corresponding to the lowest eigenvalue constant and permits consistent countable products. The resulting law is non-Gaussian under deformation, and the decomposition recovers classical Wiener chaos when the deformation is removed. We identify the associated closed gradient, adjoint divergence and deformed number operator.

The Rodrigues formula covers all orders occurring in the spectral branches, including negative orders. Under deformation, however, changing both branch orders by the same nonzero amount breaks the required relation between them. Fractional evolution defined through spectral calculus preserves the basis. After unitary identification, the one-coordinate fractional evolution operators converge in operator norm to Gaussian and three-dimensional radial Ornstein--Uhlenbeck fractional evolutions at the respective parameter endpoints. Convergence is uniform on compact time intervals away from zero. An application example with inverse stable clock is provided in the accompanying github repository.

\subclass{26A33 \and 34L10 \and 60H07 \and 60J60}
\end{abstract}

\section{Introduction}\label{sec:introduction}

Classical Wiener chaos connects Hermite polynomials, multiple stochastic
integrals and the eigenspaces of the Ornstein--Uhlenbeck operator
\cite{CameronMartin,Ito,Janson}.  In the standard normalization, the same nonnegative integer specifies the order of the Hermite polynomial, the multiplicity of the stochastic integral and the eigenvalue of the negative Ornstein--Uhlenbeck generator. Extending the Hermite polynomials to noninteger orders produces functions that are no longer polynomials and raises the question of which connections survive. This paper develops the spectral and differential aspects of this extension.

Our aim is to construct an orthogonal basis for square-integrable observables, together with operators for differentiation and evolution.
Starting from a one-dimensional eigenfunction basis, we divide its elements by the positive eigenfunction corresponding to the lowest eigenvalue and change the measure accordingly. 
This transformation makes the lowest eigenfunction constant while preserving orthogonality.
Products of the transformed eigenfunctions then form a complete orthogonal basis for square-integrable functions of the random coordinates under the resulting product probability measure. 
On this space, we construct a closed gradient and a number operator that acts diagonally on the basis.

Part~I \cite{PartI} used power-normalised parabolic-cylinder functions,
which coincide with the probabilists' Hermite polynomials at
nonnegative integer orders. At other real orders their growth on the
negative half-line prevents Gaussian square integrability. The proposed
joining procedure  \cite{PartI} also used branches normalised under different
Gaussian measures. We amend the construction by transporting both
branches into the same Hilbert space, specifying a common operator,
and matching function values and first derivatives at the origin.
The matching conditions select discrete admissible orders from the
continuous special-function family.

 The eigenvalue problem defining our one-dimensional basis
can be transferred from the Gaussian space to an unweighted
space, where it takes a form studied in
\cite{Chadzitaskos,Chadzitaskos2023}.
This connection provides a starting point for establishing
the Gaussian normalisation and completeness of the basis,
together with the self-adjoint domain of the associated
operator.

The next step addresses the product construction.
To obtain compatible expansions in countably many variables,
we need the one-coordinate basis to begin with the constant
function one. This ensures that introducing an additional
variable does not change a function that is independent
of it. In the deformed basis, however, the normalised first
eigenfunction is positive but nonconstant. Dividing every
basis function by it and using its square as a probability
density relative to the Gaussian measure gives a new
orthonormal basis whose first element is one.
This is the standard ground-state transformation
\cite[Chapter~1]{Fukushima}. The transformed operator generates a
conservative diffusion, and sums of its coordinate eigenvalues determine
the product decomposition.
We call the resulting orthogonal decomposition by total
eigenvalue \emph{spectral chaos}.

The contribution is the resulting spectral framework for this
interface family, together with its domains, parameter limits and
computable threshold coefficients. The tensor-product and
ground-state transform methods are standard, as are general non-Gaussian chaos
representations \cite{SoizeGhanem} and differentiable-measure calculi
\cite{Bogachev,DaPrato}.
At $\mu=1$, our construction recovers Gaussian chaos and
its gradient--divergence calculus \cite{Nualart}.
For $0<\mu<1$, the resulting countable product measure
and the Gaussian product measure on the same coordinate
space are mutually singular, although their
one-coordinate measures are equivalent.

Fractional calculus enters in two ways.
The fractional Rodrigues formula changes the order of
an individual special-function branch.
We justify it on a
domain that includes the negative orders in the left branch of the first eigenfunction and show why equal branch increments fail to preserve the
deformed order relation.
For evolution equations, we retain the basis and introduce
fractional effects through powers of the operator or
through a Caputo time derivative \cite{Kiryakova}. These change the time
dependence of the expansion coefficients
\cite{Balakrishnan,Schilling,Baeumer}.
As $\mu$ tends to one, the one-coordinate dynamics
converge to the Gaussian Ornstein--Uhlenbeck system;
as $\mu$ tends to zero, they converge to the radial part
of a three-dimensional Ornstein--Uhlenbeck process.
The fractional evolution operators converge in norm after the
changing Hilbert spaces have been identified, uniformly on compact
time intervals separated from zero.

Section~\ref{sec:interface} constructs the basis in the
Gaussian space and establishes its operator interpretation.
Section~\ref{sec:chaos} introduces the new probability
measure and the product spectral decomposition.
Section~\ref{sec:rodrigues} examines the fractional Rodrigues
formula and its compatibility with the construction.
Section~\ref{sec:fractional} develops the fractional
evolution and its parameter limits. As an application, we compare collective threshold correlations for
independent diffusions run on a shared inverse stable clock or on
independent copies of that clock. These time changes fit the framework
of \cite{BeghinMacciRicciuti}; the full application and numerical code
are available in online repository \cite{Code}.

\section{The Gaussian interface basis}\label{sec:interface}

\subsection{Special-function conventions}

Write $\N_0=\{0,1,2,\ldots\}$. For $y>0$, let
\begin{equation}\label{GD}
 w_y(x)=(2\pi y)^{-1/2}e^{-x^2/(2y)},\qquad
\gamma_y(\mathrm dx)=w_y(x)\mathrm dx,\qquad
\gamma=\gamma_1.
\end{equation}
Thus $\gamma_y$ is the centred Gaussian probability measure with
variance $y$. The space $L^2(\gamma_y)$ consists of measurable
functions $f:\R\to\C$ such that
\[
\|f\|_{L^2(\gamma_y)}^2
=
\int_{-\infty}^{\infty}|f(x)|^2w_y(x)\,\mathrm dx
<\infty,
\]
with functions identified when they agree $\gamma_y$-almost
everywhere. Its inner product is
\[
\langle f,g\rangle_{L^2(\gamma_y)}
=
\int_{-\infty}^{\infty}
f(x)\overline{g(x)}w_y(x)\,\mathrm dx.
\]

We use the standard parabolic-cylinder function $D_a$, with the
notation of \cite[formula~19.12.1]{AbramowitzStegun}. For
$a,z\in\C$, its hypergeometric representation is
\cite[Sections~12.2 and~12.4]{DLMF}
\[
D_a(z)
={}\sqrt{\pi}\,2^{a/2}e^{-z^2/4}
\Biggl[
\frac{{}_1F_1\!\left(-\frac a2;\frac12;\frac{z^2}{2}\right)}
     {\Gamma((1-a)/2)}
-\frac{\sqrt{2}\,z}{\Gamma(-a/2)}
{}_1F_1\!\left(\frac{1-a}{2};\frac32;\frac{z^2}{2}\right)
\Biggr].
\]
Here  
$
{}_1F_1(b;c;z)
=
\sum_{n=0}^{\infty}
\frac{(b)_n}{(c)_n}\frac{z^n}{n!}$ is Kummer's confluent hypergeometric function,
$
c\notin\{0,-1,-2,\ldots\},
$
where $(b)_0=1$ and
$(b)_n=b(b+1)\cdots(b+n-1)$ for $n\geq1$.
The reciprocal Gamma function $1/\Gamma$ is entire and vanishes
at $0,-1,-2,\ldots$, so the displayed formula also applies
when either Gamma denominator has a pole.

The power normalisation used in Part~I
\cite[Eq.~(3.8)]{PartI} is
\begin{equation}\label{eq:H}
H_a(x,y)
=
y^{a/2}e^{x^2/(4y)}D_a(x/\sqrt y),
\qquad a\in\C,\quad x\in\R,\quad y>0.
\end{equation}
For real $a$, these functions are real-valued. At nonnegative
integer orders they reduce to scaled probabilists' Hermite
polynomials:
\begin{equation}\label{eq:Pol}
H_n(x,y)=y^{n/2}\He_n(x/\sqrt y),
\qquad
\He_n(x)
=
(-1)^ne^{x^2/2}
\frac{\mathrm d^n}{\mathrm dx^n}e^{-x^2/2}.
\end{equation}
In this normalisation,
\[
\int_{-\infty}^{\infty}
\He_n(x)\He_m(x)\,\gamma(\mathrm dx)
=
n!\delta_{nm},
\qquad n,m\in\N_0;
\]
see \cite[Eq.~12.7.2 and Tables~18.3.1, 18.5.1]{DLMF}.

We record the differential equation and positive-axis
asymptotic behaviour needed below:
\begin{equation}\label{eq:Dode}
\begin{gathered}
D_a''(z)+\left(a+\frac12-\frac{z^2}{4}\right)D_a(z)=0,
\\
D_a(z)
=
z^ae^{-z^2/4}\bigl(1+O(z^{-2})\bigr)
\qquad (z\to+\infty).
\end{gathered}
\end{equation}
Here $a$ is fixed, the asymptotic limit is along the positive
real axis. These formulas follow from
\cite[Eqs.~12.2.4--12.2.5 and~12.9.1]{DLMF}.
In particular, the exponential factor in \eqref{eq:H} cancels
the Gaussian decay of $D_a$ and gives
$
H_a(x,y)
=
x^a\bigl(1+O(x^{-2})\bigr)$
for  $x\to+\infty$
for fixed $a$ and $y>0$.

For $\operatorname{Re}a<0$, an integral representation is
\begin{equation}\label{eq:Dint}
D_a(z)
=
\frac{e^{-z^2/4}}{\Gamma(-a)}
\int_0^\infty
t^{-a-1}e^{-t^2/2-zt}\,\mathrm dt,
\qquad z\in\C;
\end{equation}
see \cite[Eqs.~12.2.5 and~12.5.1]{DLMF}.
The restriction $\operatorname{Re}a<0$ ensures convergence
at $t=0$, while the Gaussian factor ensures convergence at
infinity. For $\operatorname{Re}a\geq0$, the displayed
improper integral diverges at zero; $D_a$ remains defined by
the hypergeometric formula above.

The values needed to join solutions at the origin are
\begin{equation}\label{eq:origin}
D_a(0)
=
\frac{2^{a/2}\sqrt{\pi}}{\Gamma((1-a)/2)},
\qquad
D_a'(0)
=
-\frac{2^{(a+1)/2}\sqrt{\pi}}{\Gamma(-a/2)}.
\end{equation}
They follow from
\cite[Eqs.~12.2.5--12.2.7]{DLMF}.
The differential equation and recurrence relations
\cite[Section~12.8]{DLMF}, together with \eqref{eq:H}, give
\begin{equation}\label{eq:Hidentities}
\left(\partial_y+\tfrac12\partial_{xx}\right)H_a=0,
\qquad
\partial_xH_a=aH_{a-1},
\qquad
-y\partial_{xx}H_a+x\partial_xH_a=aH_a.
\end{equation}
These are identities between smooth functions. To interpret
the last identity as an operator eigenvalue equation in
$L^2(\gamma_y)$, square integrability and membership in the
operator domain must also be established. The following
proposition identifies the square-integrability obstruction.

\begin{proposition}[Full-line integrability]
\label{prop:integrability}
For   $a{\in}\mathbb{R}$, $y{>}0$,
$
H_a(\cdot,y){\in}L^2(\gamma_y)$ if and only if $a\in\N_0$.

\end{proposition}

\begin{proof}
The function is continuous on bounded intervals, and its
power growth as $x\to+\infty$ is square integrable against
the Gaussian density. For real $a\notin\N_0$, the connection
formula and asymptotic expansions
\cite[Eqs.~12.2.5, 12.2.15, and~12.9.1--12.9.2]{DLMF}
give
\[
H_a(x,y)
=
\frac{\sqrt{2\pi}\,y^{a+1/2}}{\Gamma(-a)}
e^{x^2/(2y)}|x|^{-a-1}
\bigl(1+O(x^{-2})\bigr)
\qquad (x\to-\infty).
\]
The coefficient is nonzero because $a\notin\N_0$.
Consequently, $|H_a(x,y)|^2w_y(x)$ is asymptotic to a
positive constant times
$
e^{x^2/(2y)}|x|^{-2a-2},
$
whose integral over the negative half-line diverges.
For $a=n\in\N_0$, the function $H_n(\cdot,y)$ is a
polynomial and therefore belongs to $L^2(\gamma_y)$.
\end{proof}

This distinction also matters for the stochastic
interpretation in Part~I \cite{PartI}. For a standard Brownian motion
$W$, It\^o's formula and \eqref{eq:Hidentities} give
\[
\mathrm dH_a(W_t,t)
=
aH_{a-1}(W_t,t)\,\mathrm dW_t
\]
locally on time intervals bounded away from zero.
However, since $W_t$ has law $\gamma_t$,
Proposition~\ref{prop:integrability} implies
$
\mathbb E\bigl[|H_a(W_t,t)|^2\bigr]=\infty,$
$t>0$, $a\in\R\setminus\N_0.$
Thus the local It\^o identity does not yield a
square-integrable martingale at these orders.
Failure of square integrability alone does not settle
whether an integrable martingale interpretation is possible.

There is also a variance correction. For a nonzero
deterministic real function $g\in L^2([0,T])$, the Gaussian
random variable
\[
\xi=\int_0^T g(s)\,\mathrm dW_s\quad
{\text{has variance}}\quad
y=\int_0^T |g(s)|^2\,\mathrm ds=\|g\|_2^2.
\]
Accordingly, the matching variance argument is
$H_a(\xi,\|g\|_2^2)$.
These observations correct the variance in Definition~5
and the finite-second-moment interpretation of Theorem~7
of Part~I \cite{PartI}. At integer orders $n\geq1$, the
 martingale property and Gaussian orthogonality give
$
\operatorname{Cov}\bigl(H_n(W_s,s),H_n(W_t,t)\bigr)
=
n!s^n$,
$0<s\leq t$.

\subsection{Differential equation and matching at the origin}

Fix $0<\mu\leq1$ and $y>0$. We consider the differential equation
\begin{equation}\label{eq:DE}
\left(
-y\partial_{xx}+x\partial_x
+c_{\mu,y}x^2\ind_{\{x<0\}}
\right)\Psi_\alpha^{(\mu,y)}
=
\alpha\Psi_\alpha^{(\mu,y)},
\qquad x\neq0,
\end{equation}
where
$
c_{\mu,y}=\frac{1-\mu^2}{4\mu^2y}.
$

We seek nonzero solutions in $L^2(\gamma_y)$ that are
continuously differentiable across the origin. Here
$\alpha\in\R$ is initially a spectral parameter. The matching
conditions will select its admissible values, and the
following subsection will identify them as eigenvalues
of a self-adjoint operator.

The differential expression $(-y\partial_{xx}+x\partial_x)$
is the same on both half-lines. On the left half-line $x<0$, it is supplemented
by multiplication by $c_{\mu,y}x^2$. When $\mu=1$, this
additional term vanishes and the equation reduces to the
Hermite differential equation in \eqref{eq:Hidentities}.
The calculation below explains the choice of $c_{\mu,y}$
and constructs the solutions on the two half-lines.

On the right half-line $x>0$, the function $H_\alpha(x,y)$ satisfies the
required equation by \eqref{eq:Hidentities}. To construct
the solution on $x<0$, we begin with
$H_{\beta_\mu(\alpha)}(u,\mu y)$ for $u>0$, where
$\beta_\mu(\alpha)$ is an order to be determined. It satisfies
\[
(-\mu y\partial_{uu}+u\partial_u)
H_{\beta_\mu(\alpha)}(u,\mu y)
=
\beta_\mu(\alpha)H_{\beta_\mu(\alpha)}(u,\mu y).
\]
The positive-axis asymptotics in \eqref{eq:Dode} show that
these functions belong to $L^2((0,\infty),\gamma_y)$ and
$L^2((0,\infty),\gamma_{\mu y})$, respectively.

For either half-line $I$, we use $L^2(I,\gamma_y)$ for the
space defined using the restriction of $\gamma_y$ to $I$,
with norm
\[
\|f\|_{L^2(I,\gamma_y)}^2
=
\int_I |f(x)|^2w_y(x)\,\mathrm dx.
\]
To express the second function in the same Gaussian space
as the first, define
\begin{equation}\label{eq:rho}
\rho_{\mu,y}(u)
=
\left(\frac{w_{\mu y}(u)}{w_y(u)}\right)^{1/2}
=
\mu^{-1/4}
\exp\left[-\frac{1-\mu}{4\mu y}u^2\right].
\end{equation}
Since $\rho_{\mu,y}^2w_y=w_{\mu y}$ and $w_y$ is even,
\[
\int_{-\infty}^0
\left|\rho_{\mu,y}(-x)f(-x)\right|^2w_y(x)\,\mathrm dx
=
\int_0^\infty |f(u)|^2w_{\mu y}(u)\,\mathrm du.
\]
Consequently,
$
f(u)\longmapsto \rho_{\mu,y}(-x)f(-x)
$
defines a unitary map from
$L^2((0,\infty),\gamma_{\mu y})$ onto
$L^2((-\infty,0),\gamma_y)$.
The proposed solution on $x<0$ is therefore a constant
multiple of
$
\rho_{\mu,y}(-x)H_{\beta_\mu(\alpha)}(-x,\mu y).
$

We now determine the order $\beta_\mu(\alpha)$.
Under the substitution $u=-x>0$, the differential expression
on the negative half-line becomes
$
-y\partial_{uu}+u\partial_u+c_{\mu,y}u^2.
$
Using
\[
\frac{\rho_{\mu,y}'(u)}{\rho_{\mu,y}(u)}
=
-\frac{1-\mu}{2\mu y}u,
\]
the product rule gives, for a smooth function $f$,
\[
\begin{aligned}
&\rho_{\mu,y}^{-1}
(-y\partial_{uu}+u\partial_u+c_{\mu,y}u^2)
(\rho_{\mu,y}f)
\\
&\qquad=
-yf''+\frac{u}{\mu}f'
+\frac{1-\mu}{2\mu}f
+\left(
c_{\mu,y}-\frac{1-\mu^2}{4\mu^2y}
\right)u^2f.
\end{aligned}
\]
The chosen value of $c_{\mu,y}$ cancels the final term.
Substituting $f(u)=H_{\beta_\mu(\alpha)}(u,\mu y)$ therefore
yields
\[
\begin{aligned}
&(-y\partial_{uu}+u\partial_u+c_{\mu,y}u^2)
\left[
\rho_{\mu,y}(u)H_{\beta_\mu(\alpha)}(u,\mu y)
\right]
\\
&\qquad=
\frac{\beta_\mu(\alpha)+(1-\mu)/2}{\mu}\,
\rho_{\mu,y}(u)H_{\beta_\mu(\alpha)}(u,\mu y).
\end{aligned}
\]
For this function to satisfy the prescribed equation with
parameter $\alpha$, its order must satisfy
\begin{equation}\label{eq:order}
\beta_\mu(\alpha)
=
\mu\alpha-\frac{1-\mu}{2}.
\end{equation}
Thus the rescaling and constant shift in
\eqref{eq:order} follow from the chosen differential
expression. The multiplication by $\rho_{\mu,y}$ serves
to identify the two weighted $L^2$ spaces.

With \eqref{eq:order}, the proposed functions satisfy the
required equation on their respective half-lines.
It remains to match their values and first derivatives
at zero. By \eqref{eq:H} and \eqref{eq:rho},
\[
\rho_{\mu,y}(-x)
H_{\beta_\mu(\alpha)}(-x,\mu y)
=
\mu^{-1/4}(\mu y)^{\beta_\mu(\alpha)/2}
e^{x^2/(4y)}
D_{\beta_\mu(\alpha)}(-x/\sqrt{\mu y}).
\]
Absorbing the constant factors into real coefficients
$A_{\alpha,y}$ and $B_{\alpha,y}$ gives
\begin{equation}\label{eq:ansatz}
\Psi_\alpha^{(\mu,y)}(x)
=
e^{x^2/(4y)}
\begin{cases}
A_{\alpha,y}D_\alpha(x/\sqrt y),
&x\geq0,\\
B_{\alpha,y}D_{\beta_\mu(\alpha)}(-x/\sqrt{\mu y}),
&x<0.
\end{cases}
\end{equation}
The coefficients are to be chosen not both zero; their
dependence on the fixed parameter $\mu$ is implicit.

We impose the matching conditions
\[
\Psi_\alpha^{(\mu,y)}(0^-)
=
\Psi_\alpha^{(\mu,y)}(0^+),
\qquad
(\Psi_\alpha^{(\mu,y)})'(0^-)
=
(\Psi_\alpha^{(\mu,y)})'(0^+).
\]
These conditions ensure that the second distributional
derivative contains neither a Dirac distribution nor its
derivative at zero. The joined function \eqref{eq:ansatz} then satisfies
the differential equation \eqref{eq:DE} on the whole line in the
distributional sense.

The common exponential factor in \eqref{eq:ansatz} has
value one and derivative zero at the origin.
Differentiating the expression on $x<0$ contributes
a minus sign and a factor $1/\sqrt{\mu y}$.
The matching conditions can therefore be written as
\begin{equation}\label{eq:matrix}
\begin{pmatrix}
D_\alpha(0)&-D_{\beta_\mu(\alpha)}(0)\\
D_\alpha'(0)&\mu^{-1/2}D_{\beta_\mu(\alpha)}'(0)
\end{pmatrix}
\binom{A_{\alpha,y}}{B_{\alpha,y}}
=0.
\end{equation}
A nonzero pair of coefficients exists precisely when the
determinant vanishes. Substitution of \eqref{eq:origin} into \eqref{eq:matrix}
gives the equivalent condition $\Delta_\mu(\alpha)=0$,
where
\begin{equation}\label{eq:delta}
\Delta_\mu(\alpha)
=
\frac{\sqrt\mu}
{\Gamma(-\alpha/2)\Gamma((1-\beta_\mu(\alpha))/2)}
+
\frac{1}
{\Gamma((1-\alpha)/2)\Gamma(-\beta_\mu(\alpha)/2)}.
\end{equation}
The order relation \eqref{eq:order} ensures that the two
pieces satisfy the prescribed differential equation,
while \eqref{eq:delta} selects the parameters for which
their values and derivatives agree at zero.
The matching condition is independent of $y$.

Neither column of the matrix in \eqref{eq:matrix} is zero.
Indeed, uniqueness for the initial-value problem associated
with \eqref{eq:Dode} implies that a solution whose value
and derivative both vanish at zero is identically zero.
This cannot occur for $D_a$, by its positive-axis
asymptotic behaviour. At a root of $\Delta_\mu$, the
matrix therefore has rank one. The matching coefficients
are consequently unique up to multiplication by a common
nonzero scalar, and both coefficients are nonzero.

If $D_\alpha(0)$ and $D_{\beta_\mu(\alpha)}(0)$ are
nonzero, i.e. $\alpha,\beta_\mu(\alpha)\notin \{1,3,5,\dots\}$, one may take
\[
(A_{\alpha,y},B_{\alpha,y})
=
\bigl(D_{\beta_\mu(\alpha)}(0),D_\alpha(0)\bigr).
\]
At a root of $\Delta_\mu$, if either of these values
vanishes, the matching conditions imply that both vanish.
Their derivatives are then nonzero, and one may instead
take
\[
(A_{\alpha,y},B_{\alpha,y})
=
\bigl(
-\mu^{-1/2}D_{\beta_\mu(\alpha)}'(0),
D_\alpha'(0)
\bigr).
\]
This second case is necessary to recover the full Hermite
system. When $\mu=1$, we have $\beta_1(\alpha)=\alpha$ and
\[
\Delta_1(\alpha)
=
\frac{2}
{\Gamma(-\alpha/2)\Gamma((1-\alpha)/2)}.
\]
Its real zeros are exactly $\N_0$. Dividing the matching
conditions by the values at zero would omit all odd
Hermite polynomials, since those values vanish.

The change of Gaussian weight and the order relation above
modify the corresponding construction in Part~I
\cite[Definition~4]{PartI}. The following subsection
specifies the self-adjoint realisation of the differential
expression and proves that the normalised matched functions
form a complete orthonormal basis of eigenfunctions in
$L^2(\gamma_y)$. The proof does not rely on the full-line
completeness assertion in Part~I \cite[Theorem~9]{PartI}.

  \subsection{Closed form, domain and spectrum}
\label{section: closed form, domain and spectrum}
The functions constructed in \eqref{eq:ansatz}, with
$\Delta_\mu(\alpha)=0$, satisfy the differential equation  \eqref{eq:DE}
and the matching conditions at zero. We now specify a
self-adjoint operator for which they are eigenfunctions
and prove that their normalisations form an orthonormal
basis of $L^2(\gamma_y)$.

For an interval $I$ and $m\in\{1,2\}$, let
$W^{m,2}(I,\mathrm dx)$ denote the space of functions
whose weak derivatives up to order $m$ belong to
$L^2(I,\mathrm dx)$, with norm
\[
\|f\|_{W^{m,2}(I,\mathrm dx)}^2
=
\sum_{j=0}^m\int_I |f^{(j)}(x)|^2\,\mathrm dx,
\qquad f^{(0)}=f.
\]
The weighted space $W^{1,2}(I,\gamma_y)$ consists of
locally absolutely continuous functions $f$ such that
$f,f'\in L^2(I,\gamma_y)$, with norm
\[
\|f\|_{W^{1,2}(I,\gamma_y)}^2
=
\int_I
\bigl(|f(x)|^2+|f'(x)|^2\bigr)\,\gamma_y(\mathrm dx).
\]
We abbreviate $W^{1,2}(\R,\gamma_y)$ to
$W^{1,2}(\gamma_y)$.

For every $f\in W^{1,2}(\gamma_y)$,
\begin{equation}\label{eq:gaussianbound}  
\|xf\|_{L^2(\gamma_y)}^2
\leq
2y\|f\|_{L^2(\gamma_y)}^2
+
4y^2\|f'\|_{L^2(\gamma_y)}^2.
\end{equation}
Indeed, for $f\in C_c^\infty(\R)$, integration by parts gives
\[
\int_{-\infty}^{\infty}x^2|f(x)|^2\,\gamma_y(\mathrm dx)
=
y\|f\|_{L^2(\gamma_y)}^2
+
2y\operatorname{Re}
\int_{-\infty}^{\infty}
xf(x)\overline{f'(x)}\,\gamma_y(\mathrm dx).
\]
 Cauchy--Schwarz inequality and
\[
2y\|xf\|_{L^2(\gamma_y)}\|f'\|_{L^2(\gamma_y)}
\leq
\frac12\|xf\|_{L^2(\gamma_y)}^2
+
2y^2\|f'\|_{L^2(\gamma_y)}^2
\]
give \eqref{eq:gaussianbound}. Approximation by compactly
supported smooth functions extends the estimate to
$W^{1,2}(\gamma_y)$.

Define the sesquilinear form
\begin{equation}\label{eq:form}
\begin{aligned}
a_{\mu,y}[f,g]
&=
y\int_{-\infty}^{\infty}
f'(x)\overline{g'(x)}\,\gamma_y(\mathrm dx)
+
c_{\mu,y}\int_{-\infty}^0
x^2f(x)\overline{g(x)}\,\gamma_y(\mathrm dx),
\\
\Dom(a_{\mu,y})
&=W^{1,2}(\gamma_y).
\end{aligned}
\end{equation}
Here $a_{\mu,y}[f,g]$ denotes the value of the form at
$(f,g)$, and $a_{\mu,y}[f,f]$ is its quadratic value.
Since $y\geq0$, $c_{\mu,y}\geq0$, the form \eqref{eq:form} is nonnegative.
Estimate \eqref{eq:gaussianbound} shows that the form norm
\[
\left(
\|f\|_{L^2(\gamma_y)}^2+a_{\mu,y}[f,f]
\right)^{1/2}
\]
is equivalent to $\|f\|_{W^{1,2}(\gamma_y)}$, with constants
depending on $\mu$ and $y$. The form is therefore closed.

The representation theorem for densely defined closed
nonnegative forms \cite[Chapter~VI, Section~2]{Kato}
determines a unique nonnegative self-adjoint operator
$A_{\mu,y}$. Its domain consists of those
$f\in W^{1,2}(\gamma_y)$ for which there exists
$h\in L^2(\gamma_y)$ satisfying
\[
a_{\mu,y}[f,g]
=
\langle h,g\rangle_{L^2(\gamma_y)}
\qquad
\text{for every }g\in W^{1,2}(\gamma_y).
\]
For such $f$, the vector $h$ is unique and
$A_{\mu,y}f=h$.

\begin{theorem}[Self-adjoint realisation and spectral basis]
\label{thm:spectrum}
The operator determined by \eqref{eq:form} acts as
\begin{equation}\label{eq:A}
A_{\mu,y}f
=
-yf''+xf'
+c_{\mu,y}x^2\ind_{\{x<0\}}f.
\end{equation}
Its domain consists of the functions
$f\in W^{1,2}(\gamma_y)$ such that, for every $R>0$,
\[
f|_{(-R,0)}\in W^{2,2}((-R,0),\mathrm dx),
\qquad
f|_{(0,R)}\in W^{2,2}((0,R),\mathrm dx),
\]
the expression in \eqref{eq:A} belongs to $L^2(\gamma_y)$,
and
$
f'(0^-)=f'(0^+).
$

The resolvent $(A_{\mu,y}+I)^{-1}$ is compact.
The eigenvalues of $A_{\mu,y}$ are independent of $y$ and form a simple
sequence
\[
0\leq\alpha_0(\mu)<\alpha_1(\mu)<\cdots
\longrightarrow\infty.
\]
A real number $\alpha$ is an eigenvalue of $A_{\mu,y}$ if and only if
$\Delta_\mu(\alpha)=0$, where $\Delta_\mu$ is defined
in \eqref{eq:delta}.

For $k\in\N_0$, 
\begin{equation}\label{eq:EigenA_mu_y}
A_{\mu,y}e_k^{(\mu,y)}
=
\alpha_k(\mu)e_k^{(\mu,y)},\qquad e_k^{(\mu,y)}
=
\frac{\Psi_{\alpha_k(\mu)}^{(\mu,y)}}
{\bigl\|\Psi_{\alpha_k(\mu)}^{(\mu,y)}
\bigr\|_{L^2(\gamma_y)}}.
\end{equation}

These real-valued functions $e_k^{(\mu,y)}$ form an orthonormal basis
of $L^2(\gamma_y)$. The eigenfunction $e_k^{(\mu,y)}$ corresponding to the
smallest eigenvalue is strictly positive:
\[
e_0^{(\mu,y)}(x)>0
\qquad\text{for every }x\in\R.
\]
\end{theorem}

\begin{proof}
We first identify the domain of $A_{\mu,y}$. Taking a test function
$g\in C_c^\infty(\R)$ supported in either open half-line
in the form identity gives \eqref{eq:A} in distributions.
On bounded intervals, the Gaussian density is bounded
above and below by positive constants. Since $f,f'$ and
$A_{\mu,y}f$ are square integrable there, the equation \eqref{eq:DE}
implies the stated $W^{2,2}$ regularity on both sides
of zero.

For a compactly supported smooth test function whose
support crosses zero, integration by parts gives
\[
a_{\mu,y}[f,g]
=
\langle A_{\mu,y}f,g\rangle_{L^2(\gamma_y)}
+
yw_y(0)
\bigl(f'(0^-)-f'(0^+)\bigr)\overline{g(0)}.
\]
The form identity therefore requires
$f'(0^-)=f'(0^+)$. Conversely, the stated domain
conditions give the form identity for every
$g\in C_c^\infty(\R)$. Density in the form norm extends
it to every $g\in W^{1,2}(\gamma_y)$.

To prove compactness of $(A_{\mu,y}+I)^{-1}$, consider the unitary map
\[
J_y:L^2(\gamma_y)\longrightarrow L^2(\R,\mathrm dx),
\qquad
J_yf=w_y^{1/2}f.
\]
Direct differentiation gives
\begin{equation}\label{eq:schrodinger}
J_yA_{\mu,y}J_y^{-1}
=
-y\partial_{xx}-\frac12
+
\begin{cases}
x^2/(4\mu^2y),&x<0,\\
x^2/(4y),&x\geq0,
\end{cases}
\end{equation}
where the final term acts by multiplication. Moreover,
\[
J_y\bigl(W^{1,2}(\gamma_y)\bigr)
=
\left\{
u\in W^{1,2}(\R,\mathrm dx):
xu\in L^2(\R,\mathrm dx)
\right\}.
\]
This follows from \eqref{eq:gaussianbound} and
\[
y\|f'\|_{L^2(\gamma_y)}^2
=
y\|(J_yf)'\|_{L^2(\R,\mathrm dx)}^2
+
\frac1{4y}\|xJ_yf\|_{L^2(\R,\mathrm dx)}^2
-\frac12\|J_yf\|_{L^2(\R,\mathrm dx)}^2,
\]
first for compactly supported smooth functions and then
by approximation.

A sequence bounded in the form norm has images under
$J_y$ bounded in $W^{1,2}(\R,\mathrm dx)$, with
$\|xJ_yf\|_{L^2(\R,\mathrm dx)}$ uniformly bounded.
Rellich's theorem \cite{Rel} gives compactness on each bounded
interval, while
\[
\int_{|x|>R}|(J_yf)(x)|^2\,\mathrm dx
\leq
R^{-2}\|xJ_yf\|_{L^2(\R,\mathrm dx)}^2
\]
controls the remaining part of the integral uniformly.
The form domain therefore embeds compactly into
$L^2(\gamma_y)$, proving compactness of
$(A_{\mu,y}+I)^{-1}$.

On each half-line, the square-integrable solution of
the transformed differential equation is unique up to
a constant factor. The parabolic-cylinder asymptotics
identify these solutions as
\[
D_\alpha(x/\sqrt y)
\quad (x>0),
\qquad
D_{\beta_\mu(\alpha)}(-x/\sqrt{\mu y})
\quad (x<0).
\]
Multiplication by $w_y^{1/2}$ preserves the matching
conditions because $w_y^{1/2}(0)>0$ and
$(w_y^{1/2})'(0)=0$. Hence every eigenfunction has the
form \eqref{eq:ansatz}, with coefficients satisfying
\eqref{eq:matrix}. Conversely, the asymptotic behaviour
and matching conditions show that the functions so
constructed belong to the stated operator domain.
Thus $\alpha$ is an eigenvalue exactly when
$\Delta_\mu(\alpha)=0$.

At each such value, the matching matrix has a
one-dimensional nullspace, so the eigenvalue is simple.
The spectral theorem for self-adjoint operators with
compact resolvent gives completeness.

The smallest eigenvalue is the minimum of
\[
\frac{a_{\mu,y}[f,f]}{\|f\|_{L^2(\gamma_y)}^2},
\qquad f\neq0.
\]
Replacing a minimiser by its modulus preserves its norm
and does not increase the form value. We may therefore
choose a nonnegative eigenfunction for this eigenvalue.
If this continuously differentiable function vanished
at a point, its derivative would vanish there as well.
Uniqueness for the differential equation would then
make it identically zero. It is therefore strictly
positive.

Finally, the unitary change of variables
$f(x)\mapsto f(\sqrt y\,z)$ from $L^2(\gamma_y)$ to
$L^2(\gamma)$ identifies $A_{\mu,y}$ with $A_{\mu,1}$.
Thus the eigenvalues do not depend on $y$.
\end{proof}

As $\mu$ decreases, the coefficient $c_{\mu,y}$ increases,
and the corresponding quadratic forms are ordered.
The limiting comparison on $(0,\infty)$ uses the
differential expression $(-y\partial_{xx}+x\partial_x)$
with the boundary condition $f(0)=0$.

\begin{proposition}[Spectral bounds and endpoint limits]
\label{prop:limits}
For $k\in\N_0$,
\begin{equation}\label{eq:bounds}
k\leq\alpha_k(\mu)\leq2k+1,
\qquad
\lim_{\mu\uparrow1}\alpha_k(\mu)=k,
\qquad
\lim_{\mu\downarrow0}\alpha_k(\mu)=2k+1.
\end{equation}
Each eigenvalue $\alpha_k(\mu)$ of operator \eqref{eq:A} is real analytic on $(0,1)$ and strictly
decreasing on $(0,1]$. In particular,
$0<\alpha_0(\mu)<1$,
$0<\mu<1$.
\end{proposition}

\begin{proof}
The inequality
$a_{\mu,y}[f,f]\geq a_{1,y}[f,f]$ and the min--max
principle give $\alpha_k(\mu)\geq k$.

For the upper bound, restrict
$
H_1(\cdot,y),H_3(\cdot,y),\ldots,H_{2k+1}(\cdot,y),
$
given by \eqref{eq:Pol}, to $(0,\infty)$ and extend them by zero to
$(-\infty,0)$. Since these polynomials vanish at zero,
their extensions belong to $W^{1,2}(\gamma_y)$.
Their span has dimension $(k+1)$ and the largest Rayleigh
quotient $(2k+1)$. The min--max principle gives
$\alpha_k(\mu)\leq2k+1$.

The forms $a_{\mu,y}[f,f]$ have a common domain for different values of $\mu$ and depend analytically
on $\mu\in(0,1)$. Simplicity of the eigenvalues and
analytic perturbation theory
\cite[Chapter~VII, Section~4]{Kato} give real analyticity.
Differentiating the eigenvalue identity (\ref{eq:EigenA_mu_y}) with the
eigenfunction normalised in $L^2(\gamma_y)$ gives
\[
\alpha_k'(\mu)
=
-\frac1{2\mu^3y}
\int_{-\infty}^0
x^2|e_k^{(\mu,y)}(x)|^2\,\gamma_y(\mathrm dx)
<0.
\]
The strict inequality follows because an eigenfunction
cannot vanish on the entire negative half-line.
Thus each eigenvalue $\alpha_k(\mu)$ is strictly decreasing  as a function of $\mu$.

As $\mu\uparrow1$, we have $c_{\mu,y}\to0$.
Applying the min--max principle to the span of
$H_0(\cdot,y),\ldots,H_k(\cdot,y)$ gives
$\limsup_{\mu\uparrow1}\alpha_k(\mu)\leq k$.
Together with the lower bound, this proves the first
limit in \eqref{eq:bounds}.

For the limit as $\mu\downarrow0$, the first $(k+1)$
normalised eigenfunctions are uniformly bounded in
$W^{1,2}(\gamma_y)$, since their eigenvalues are at most
$(2k+1)$. Along any sequence tending to zero, compactness
gives a subsequence converging strongly in
$L^2(\gamma_y)$ and weakly in $W^{1,2}(\gamma_y)$ to
functions $f_0,\ldots,f_k$. The limits remain orthonormal.
Moreover,
\[
c_{\mu,y}\int_{-\infty}^0
x^2|e_j^{(\mu,y)}(x)|^2\,\gamma_y(\mathrm dx)
\leq\alpha_j(\mu).
\]
Since $c_{\mu,y}\to\infty$, each limit vanishes on
$(-\infty,0)$ and, by continuity, has value zero at
the origin.

The limiting comparison form is
\[
f\longmapsto
y\int_0^\infty |f'(x)|^2\,\gamma_y(\mathrm dx),
\qquad
f\in W^{1,2}((0,\infty),\gamma_y),\quad f(0)=0.
\]
The self-adjoint operator  $(-y\partial_{xx}+x\partial_x)$ is associated with $A_{1,y}$ in \eqref{eq:A}, acting in
$L^2((0,\infty),\gamma_y)$ with boundary condition $f(0)=0$.
Odd extension identifies its eigenvalue problem with that
of $A_{1,y}$ restricted to the odd subspace of $L^2(\gamma_y)$.
Its eigenfunctions are therefore the restrictions of
$
H_1(\cdot,y),H_3(\cdot,y),H_5(\cdot,y),\ldots
$
to $(0,\infty)$, with respective eigenvalues $1,3,5,\ldots$.

For fixed coefficients $b_0,\ldots,b_k$, weak lower
semicontinuity gives
\[
y\int_0^\infty
\left|\sum_{j=0}^k b_jf_j'(x)\right|^2
\,\gamma_y(\mathrm dx)
\leq
\left(\lim_{\mu\downarrow0}\alpha_k(\mu)\right)
\sum_{j=0}^k|b_j|^2.
\]
Applying the min--max principle to the span of the
limits yields
\[
2k+1\leq\lim_{\mu\downarrow0}\alpha_k(\mu).
\]
The upper bound proves equality.

If $\mu<1$ and $\alpha_0(\mu)=0$, then
$a_{\mu,y}[e_0^{(\mu,y)},e_0^{(\mu,y)}]=0$ would force
$e_0^{(\mu,y)}$ to be constant and to vanish on the
negative half-line, contradicting its normalisation.
The strict upper bound $\alpha_0(\mu)<1$ follows from
strict monotonicity and \eqref{eq:bounds}.
\end{proof}

We henceforth write
$
A_\mu=A_{\mu,1}$,
$
e_k^{(\mu)}=e_k^{(\mu,1)}.
$
At $\mu=1$, we choose the signs so that
\[
e_k^{(1)}=\frac{\He_k{(x)}}{\sqrt{k!} } = \frac{H_k{(x,1)}}{\sqrt{k!} }.
\]

If $A_{\mu,y}f=\alpha f$, the matching conditions and
\eqref{eq:A} imply 
\[
f''(0^\pm)=-\frac{\alpha}{y}f(0),
\qquad
f'''(0^\pm)=\frac{1-\alpha}{y}f'(0).
\]
Thus eigenfunctions are $C^3$ across zero. These
identities follow from the eigenvalue equation \eqref{eq:EigenA_mu_y} and
do not impose additional conditions on the operator
domain.

\subsection{A Brownian interpretation of the eigenfunctions}

For $k\in\N_0$, define
\[
\Psi_k^{(\mu)}(x,t)
=
t^{\alpha_k(\mu)/2}
e_k^{(\mu)}(x/\sqrt t),
\qquad x\in\R,\quad t>0.
\]
The $C^3$-regularity of the eigenfunctions $e_k^{(\mu)}(y)$ across zero,
established at the end of subsection \ref{section: closed form, domain and spectrum},
ensures that $\Psi_k^{(\mu)}$ is continuously
differentiable in $t$ and twice continuously
differentiable in $x$ on $\R\times(0,\infty)$.

Differentiation and the eigenvalue equation
\eqref{eq:EigenA_mu_y} for $A_{\mu,y}$ with $y=1$, 
give
\[
\left(
\partial_t+\tfrac12\partial_{xx}
-\frac{c_{\mu,1}x^2}{2t^2}\ind_{\{x<0\}}
\right)\Psi_k^{(\mu)}=0.
\]

Let $W$ be a standard Brownian motion and fix $0<s<T$.
The restriction $s>0$ keeps the argument away from the
singularity at $t=0$.
Then
\[
\exp\left[
-\frac{c_{\mu,1}}2
\int_s^t
\frac{W_r^2}{r^2}\ind_{\{W_r<0\}}\,\mathrm dr
\right]
\Psi_k^{(\mu)}(W_t,t),
\qquad s\leq t\leq T,
\]
is a square-integrable martingale with respect to the
Brownian filtration.

Indeed, It\^o's formula and the product rule cancel
the finite-variation terms. 
 
More precisely, for $s\leq t\leq T$,
\[
\begin{aligned}
&\exp\left[
-\frac{c_{\mu,1}}2
\int_s^t
\frac{W_r^2}{r^2}\ind_{\{W_r<0\}}\,\mathrm dr
\right]
\Psi_k^{(\mu)}(W_t,t)
\\
&\quad=
\Psi_k^{(\mu)}(W_s,s)
+
\int_s^t
\exp\left[
-\frac{c_{\mu,1}}2
\int_s^r
\frac{W_u^2}{u^2}\ind_{\{W_u<0\}}\,\mathrm du
\right]
\partial_x\Psi_k^{(\mu)}(W_r,r)\,\mathrm dW_r.
\end{aligned}
\]

Taking into account that $c_{\mu,1}=(1-\mu^2)/(4\mu^2)\geq0$ for
$0<\mu\leq1$, the exponential factor is at most one.
Moreover, $W_t/\sqrt{t}$ has distribution $\gamma$,
so the normalisation of $e_k^{(\mu)}$ and the fact that
$e_k^{(\mu)}\in W^{1,2}(\gamma)$ give the following
estimates:

\[
 \mathbb E\int_s^T
\left|
\partial_x\Psi_k^{(\mu)}(W_t,t)
\right|^2\,\mathrm dt
=
\|(e_k^{(\mu)})'\|_{L^2(\gamma)}^2
\int_s^T t^{\alpha_k(\mu)-1}\,\mathrm dt
<\infty,
\]
since
\[
\partial_x\Psi_k^{(\mu)}(x,t)
=
t^{(\alpha_k(\mu)-1)/2}
(e_k^{(\mu)})'(x/\sqrt{t}).
\]

Also,
\[
\mathbb E\left|\Psi_k^{(\mu)}(W_t,t)\right|^2
=
t^{\alpha_k(\mu)}.
\]
Hence the initial value $\Psi_k^{(\mu)}(W_s,s)$ of the
process $\Psi_k^{(\mu)}$ is square integrable, and the It\^o
integral in the stochastic representation above is a
square-integrable martingale. This proves the asserted
martingale property.

\section{Weighted calculus and product spectral chaos}
\label{sec:chaos}

\subsection{Why the one-coordinate law changes}

The basis constructed in Theorem~\ref{thm:spectrum} gives expansions for functions
of finitely many independent Gaussian variables.
Let $g_1,\ldots,g_m$ be nonzero orthogonal real functions
in $L^2([0,T])$ and set
\[
\xi_j=\int_0^T g_j(t)\,\mathrm dW_t,
\qquad
y_j=\|g_j\|_2^2.
\]

The variables $\xi_j$ are independent and have laws 
$\gamma_{y_j}$. Therefore,   
\[
\left\{
\prod_{j=1}^m e_{k_j}^{(\mu,y_j)}(\xi_j):
(k_1,\ldots,k_m)\in\N_0^m
\right\}
\]
is an orthonormal basis of
$L^2(\sigma(\xi_1,\ldots,\xi_m))$.

To index products by sequences with finitely many nonzero
entries, we require a zero entry to indicate that the
corresponding coordinate does not appear in the function.
For the product formula to have this property, the zeroth
one-coordinate basis function must be the constant one.

For example, the random variable
$e_k^{(\mu,y_1)}(\xi_1)$ remains unchanged when a second
coordinate $\xi_2$ is introduced: its value still
depends only on $\xi_1$.
However, the index $(k,0)$ in the two-coordinate
product basis represents
$
e_k^{(\mu,y_1)}(\xi_1)e_0^{(\mu,y_2)}(\xi_2).
$
For $0<\mu<1$, the second factor is nonconstant,
so this is a different random variable.
Thus appending a zero index leaves a basis function
unchanged only when the zeroth one-coordinate
basis function is $1$.

We therefore modify the one-coordinate basis so that
its zeroth element is $1$. By the scaling in
Theorem~\ref{thm:spectrum}, it is enough to work with
variance one. Define
\begin{equation}\label{eq:ground}
\nu_\mu(\mathrm dx)
=
(e_0^{(\mu)}(x))^2\gamma(\mathrm dx),
\qquad
r_k^{(\mu)}
=
\frac{e_k^{(\mu)}}{e_0^{(\mu)}},\qquad e_k^{(\mu)}=e_k^{(\mu,1)}.
\end{equation}
The strict positivity of $e_0^{(\mu)}$ makes the quotients in \eqref{eq:ground}
well defined. Its normalisation gives $\nu_\mu(\R)=1$,
and $r_0^{(\mu)}=1$. Moreover,
\begin{equation}\label{eq:r-orthonormality}
\int_{-\infty}^\infty
r_k^{(\mu)}(x)r_j^{(\mu)}(x)\,\nu_\mu(\mathrm dx)
=
\int_{-\infty}^\infty
e_k^{(\mu)}(x)e_j^{(\mu)}(x)\,\gamma(\mathrm dx)
=
\delta_{kj}.
\end{equation}

Multiplication by $e_0^{(\mu)}$ defines a unitary map
\begin{equation}\label{eq:Gmu}
G_\mu:L^2(\nu_\mu)\longrightarrow L^2(\gamma),
\qquad
(G_\mu f)(x)=e_0^{(\mu)}(x)f(x).
\end{equation}
Since $G_\mu r_k^{(\mu)}=e_k^{(\mu)}$, the functions
$r_k^{(\mu)}$ form a complete orthonormal basis of
$L^2(\nu_\mu)$.

\begin{definition}\label{NumOp}
Define the {\bf deformed number operator}
$$
N_\mu
=
G_\mu^{-1}(A_\mu-\alpha_0(\mu)I)G_\mu,
\qquad
\lambda_k(\mu)=\alpha_k(\mu)-\alpha_0(\mu),
$$
where $A_\mu=A_{\mu,1}$, $A_{\mu,y}$ is given by \eqref{eq:A}.

Expanding the derivatives gives
\begin{equation}\label{eq:Nmu}
N_\mu f
=
-f''
+
\left(
x-2\frac{(e_0^{(\mu)})'}{e_0^{(\mu)}}
\right)f'.
\end{equation}
\end{definition}

This is a nonnegative self-adjoint operator with domain
\[
\Dom(N_\mu)
=
\left\{
f\in L^2(\nu_\mu):
e_0^{(\mu)}f\in\Dom(A_\mu)
\right\}.
\]
It satisfies
$
N_\mu r_k^{(\mu)}
=
\lambda_k(\mu)r_k^{(\mu)}$,
$
0=\lambda_0(\mu)<\lambda_1(\mu)<\cdots.
$
In particular, $N_\mu1{=}0$. The integer $k$ indexes the
basis functions, while $\lambda_k(\mu)$ is the
corresponding eigenvalue of $N_\mu$ and $\lambda_k(\mu)$ does not need to be an integer.

\subsection{The gradient and conservative diffusion}

We now identify $N_\mu$ in terms of differentiation
in $L^2(\nu_\mu)$. This gives a closed derivative,
its adjoint and the generator of the associated
Markov semigroup.

\begin{proposition}[Closed gradient and divergence]
\label{prop:gradient}
The derivative $f\mapsto f'$ on $C_c^\infty(\R)$ is
closable in $L^2(\nu_\mu)$. Its closure $D_\mu$ has domain
\[
\Dom(D_\mu)=W^{1,2}(\nu_\mu),
\]
where $W^{1,2}(\nu_\mu)=W^{1,2}(\R,\nu_\mu)$ consists
of locally absolutely continuous functions $f$ such
that $f,f'\in L^2(\nu_\mu)$, with norm
\[
\|f\|_{W^{1,2}(\nu_\mu)}^2
=
\int_{-\infty}^{\infty}
\bigl(|f(x)|^2+|f'(x)|^2\bigr)\,\nu_\mu(\mathrm dx).
\]
Furthermore,
$
\Dom(N_\mu^{1/2})=\Dom(D_\mu)$,
$
N_\mu=D_\mu^*D_\mu.
$

The divergence $\delta_\mu=D_\mu^*$ satisfies
\begin{equation}\label{eq:divergence}
\delta_\mu g
=
-g'
+
\left(
x-2\frac{(e_0^{(\mu)})'}{e_0^{(\mu)}}
\right)g
\end{equation}
for $g\in C_c^\infty(\R)$.

The domain of $N_\mu$ consists of
$f\in W^{1,2}(\nu_\mu)$ whose derivative $f'$ is locally
absolutely continuous and for which
\begin{equation}\label{eq:Ndiff}
-f''
+
\left(
x-2\frac{(e_0^{(\mu)})'}{e_0^{(\mu)}}
\right)f'
\in L^2(\nu_\mu).
\end{equation}
On this domain, the expression in \eqref{eq:Ndiff}
equals $N_\mu f$.
\end{proposition}

\begin{proof}
Write
$
m_\mu(x)=(e_0^{(\mu)}(x))^2w_1(x)$,
$
\nu_\mu(\mathrm dx)=m_\mu(x)\,\mathrm dx.
$
The density is positive and continuously differentiable.
The functions $m_\mu$, $m_\mu^{-1}$ and $m_\mu'/m_\mu$
are bounded on compact intervals.

Suppose $f_n\in C_c^\infty(\R)$,
$f_n\to0$ and $f_n'\to g$ in $L^2(\nu_\mu)$.
For every $\phi\in C_c^\infty(\R)$, integration by parts
gives
\[
\langle g,\phi\rangle_{L^2(\nu_\mu)}
=
\lim_{n\to\infty}
\langle f_n',\phi\rangle_{L^2(\nu_\mu)}
=
\lim_{n\to\infty}
\left\langle
f_n,-\phi'-\frac{m_\mu'}{m_\mu}\phi
\right\rangle_{L^2(\nu_\mu)}
=0.
\]
Thus $g=0$, proving closability. Cutoff functions and
local mollification show that $C_c^\infty(\R)$ is dense
in $W^{1,2}(\nu_\mu)$ in the displayed Sobolev norm.
This identifies the domain of the closed derivative.

Using
$A_\mu e_0^{(\mu)}=\alpha_0(\mu)e_0^{(\mu)}$,
expanding derivatives and integrating by parts gives,
for compactly supported
$f,g\in W^{1,2}(\R,\mathrm dx)$,
\begin{equation}\label{eq:groundidentity}
a_{\mu,1}[e_0^{(\mu)}f,e_0^{(\mu)}g]
-
\alpha_0(\mu)
\langle e_0^{(\mu)}f,e_0^{(\mu)}g\rangle_{L^2(\gamma)}
= \int_{-\infty}^{\infty}
f'(x)\overline{g'(x)}\,\nu_\mu(\mathrm dx).
\end{equation}
The left-hand side is the closed form of $N_\mu$,
obtained by unitary conjugation.

Compactly supported smooth functions are dense in
this form domain. Indeed, if
$e_0^{(\mu)}f\in W^{1,2}(\gamma)$, choose
$u_n\in C_c^\infty(\R)$ converging to
$e_0^{(\mu)}f$ in the original form norm.
Then $u_n/e_0^{(\mu)}$ converges to $f$ in the form norm
of $N_\mu$. Each quotient has compact support and
belongs to $W^{1,2}(\R,\mathrm dx)$.
Since the density is bounded above and below on compact
intervals, local mollification approximates these
quotients by compactly supported smooth functions in
the weighted Sobolev norm.
Identity \eqref{eq:groundidentity} gives the same
approximation in the form norm of $N_\mu$.

Consequently, the form of $N_\mu$ is the closed
derivative form
\[
(f,g)\longmapsto
\langle D_\mu f,D_\mu g\rangle_{L^2(\nu_\mu)}.
\]
This proves $N_\mu=D_\mu^*D_\mu$ and
$\Dom(N_\mu^{1/2})=\Dom(D_\mu)$.

Integration by parts gives
\[
D_\mu^*g=-g'-\frac{m_\mu'}{m_\mu}g
\]
for compactly supported smooth $g$. Since
\[
\frac{m_\mu'}{m_\mu}
=
2\frac{(e_0^{(\mu)})'}{e_0^{(\mu)}}-x,
\]
we obtain \eqref{eq:divergence}.

Finally, the form identity for $N_\mu f=h$ gives
$-(m_\mu f')'=m_\mu h$
in distributions. Hence $m_\mu f'$, and therefore
$f'$, is locally absolutely continuous.
Division by $m_\mu$ gives \eqref{eq:Ndiff}.
Conversely, the stated regularity and integrability
conditions imply the form identity for compactly
supported smooth tests and then, by density, for
all $g\in W^{1,2}(\nu_\mu)$. This proves the domain
description.
\end{proof}

The derivative form has the Markov contraction property:
replacing a real function $f$ by $(0\vee f)\wedge1$ does
not increase the integral of its squared derivative.
Consequently,
\[
e^{-tN_\mu/2}
=
e^{\alpha_0(\mu)t/2}
G_\mu^{-1}e^{-tA_\mu/2}G_\mu
\]
is a symmetric Markov semigroup. Since $N_\mu1=0$,
$e^{-tN_\mu/2}1=1.$
 
Thus the semigroup $(e^{-tN_\mu/2})_{t\geq0}$ is
conservative, and $\nu_\mu$ is invariant and
reversible for this semigroup.

Its generator is
\[
-\tfrac12N_\mu
=
\tfrac12\partial_{xx}
+
\left(
\frac{(e_0^{(\mu)})'}{e_0^{(\mu)}}-\frac x2
\right)\partial_x.
\]

The semigroup $(e^{-tN_\mu/2})_{t\geq0}$ is the
transition semigroup of the diffusion
$(X_t)_{t\geq0}$ solving
\[
\mathrm dX_t
=
\mathrm dW_t
+
\left(
\frac{(e_0^{(\mu)})'(X_t)}{e_0^{(\mu)}(X_t)}
-\frac{X_t}{2}
\right)\mathrm dt.
\]

The drift coefficient is the function
\[
x\longmapsto
\frac{(e_0^{(\mu)})'(x)}{e_0^{(\mu)}(x)}-\frac{x}{2}.
\]

Recall that $e_0^{(\mu)}$ is the strictly positive,
$L^2(\gamma)$-normalised eigenfunction corresponding
to $\alpha_0(\mu)$. Its parabolic-cylinder representation is
\[
e_0^{(\mu)}(x)
=
e_0^{(\mu)}(0)e^{x^2/4}
\begin{cases}
\dfrac{D_{\alpha_0(\mu)}(x)}
      {D_{\alpha_0(\mu)}(0)},&x\geq0,\\[8pt]
\dfrac{D_{\beta_\mu(\alpha_0(\mu))}(-x/\sqrt{\mu})}
      {D_{\beta_\mu(\alpha_0(\mu))}(0)},&x<0,
\end{cases}
\]
where $\beta_\mu(a)=\mu a-(1-\mu)/2$.
The positive constant $e_0^{(\mu)}(0)$ is determined
by $\|e_0^{(\mu)}\|_{L^2(\gamma)}=1$ and cancels
in the ratio $(e_0^{(\mu)})'/e_0^{(\mu)}$.

Since $e_0^{(\mu)}$ is strictly positive and $C^3$,
this function has a continuous derivative, bounded
on every compact interval. It is therefore locally
Lipschitz.

 The parabolic-cylinder derivative
identities and asymptotic formulas
\cite[Sections~12.8--12.9]{DLMF} give
\[
\frac{(e_0^{(\mu)})'(x)}{e_0^{(\mu)}(x)}-\frac x2
=
\begin{cases}
-x/2+O(x^{-1}),&x\to+\infty,\\
-x/(2\mu)+O(|x|^{-1}),&x\to-\infty.
\end{cases}
\]
The drift has at most linear growth, so the stochastic
differential equation has a unique nonexplosive solution.

The relation between closed derivative forms and
symmetric Markov semigroups is standard
\cite[Chapter~1]{Fukushima}. The preceding calculations
identify the derivative, its adjoint and the operator
domain for $\nu_\mu$, giving a weighted
gradient--divergence calculus within the
differentiable-measure framework
\cite{Bogachev,DaPrato}. At $\mu=1$, we have
$e_0^{(1)}=1$ and $\nu_1=\gamma$, so the construction
reduces to Gaussian calculus. For $\mu<1$,
$\delta_\mu$ is defined as the adjoint in
$L^2(\nu_\mu)$; its identification with a Wiener-space
Skorohod integral is not part of this construction.

\subsection{The product gradient and spectral decomposition}

Equip $\R^\N$ with its product Borel sigma-algebra
and probability measure
$\mathbb P_\mu{\,=\,}\nu_\mu^{\otimes\N}.$
Let $X_j(x)=x_j$ be the coordinate maps.
For a multi-index $\mathbf k{\,=\,}(k_j)_{j\geq1}$ with
$k_j\in\N_0$ and only finitely many nonzero entries,
define
\begin{equation}\label{eq:prod}
R_{\mathbf k}^{(\mu)}
=
\prod_{j\geq1}r_{k_j}^{(\mu)}(X_j),
\qquad
\Lambda_{\mathbf k}(\mu)
=
\sum_{j\geq1}\lambda_{k_j}(\mu).
\end{equation}

In \eqref{eq:prod}, the product contains only finitely
many factors different from one, and the sum contains
only finitely many nonzero terms, because
$\mathbf k$ has finite lenght,
$r_0^{(\mu)}=1$ and $\lambda_0(\mu)=0$.

On the finite linear span of the functions
$R_{\mathbf k}^{(\mu)}$, define
\[
(\mathbf D_\mu^0R_{\mathbf k}^{(\mu)})_j
=
(r_{k_j}^{(\mu)})'(X_j)
\prod_{i\neq j}r_{k_i}^{(\mu)}(X_i).
\]
 The superscript $0$ in $\mathbf D_\mu^0$ distinguishes the gradient
initially defined on finite linear combinations
of the product basis functions from its closure,
denoted below by $\mathbf D_\mu$.

The coordinate gradient $\mathbf D_\mu^0$ takes values in
$L^2(\mathbb P_\mu;\ell^2)$, whose norm is
\[
\|V\|_{L^2(\mathbb P_\mu;\ell^2)}^2
=
\int_{\R^\N}
\sum_{j\geq1}|V_j|^2\,\mathrm d\mathbb P_\mu.
\]

\begin{theorem}[Basis and decomposition]
\label{thm:product}
The functions $R_{\mathbf k}^{(\mu)}$ form an
orthonormal basis of $L^2(\mathbb P_\mu)$.

For $F\in L^2(\mathbb P_\mu)$,
the next decomposition is valid
\[
F=\sum_{\mathbf k}
F_{\mathbf k}R_{\mathbf k}^{(\mu)}
\qquad\text{in }\qquad L^2(\mathbb P_\mu),
\]
where
\[
F_{\mathbf k}
=
\langle F,R_{\mathbf k}^{(\mu)}
\rangle_{L^2(\mathbb P_\mu)},\qquad \|F\|_{L^2(\mathbb P_\mu)}^2
=
\sum_{\mathbf k}|F_{\mathbf k}|^2.
\]

The operator $\mathbf D_\mu^0$ is closable.
Its closure $\mathbf D_\mu$ has domain
\[
\Dom(\mathbf D_\mu)
=
\left\{
F\in L^2(\mathbb P_\mu):
\sum_{\mathbf k}
\Lambda_{\mathbf k}(\mu)|F_{\mathbf k}|^2<\infty
\right\}.
\]
The nonnegative self-adjoint operator
$
\mathbf N_\mu=\mathbf D_\mu^*\mathbf D_\mu
$
has domain
\[
\Dom(\mathbf N_\mu)
=
\left\{
F\in L^2(\mathbb P_\mu):
\sum_{\mathbf k}
\Lambda_{\mathbf k}(\mu)^2|F_{\mathbf k}|^2<\infty
\right\}.
\]
Moreover,
\[
\|\mathbf D_\mu F\|_{L^2(\mathbb P_\mu;\ell^2)}^2
=
\sum_{\mathbf k}
\Lambda_{\mathbf k}(\mu)|F_{\mathbf k}|^2,
\]
and
\[
\mathbf N_\mu F
=
\sum_{\mathbf k}
\Lambda_{\mathbf k}(\mu)F_{\mathbf k}
R_{\mathbf k}^{(\mu)},
\qquad
F\in\Dom(\mathbf N_\mu).
\]
In particular,
$
\mathbf N_\mu R_{\mathbf k}^{(\mu)}
=
\Lambda_{\mathbf k}(\mu)R_{\mathbf k}^{(\mu)}.
$
\end{theorem}

\begin{proof}
Independence under
$\mathbb P_\mu=\nu_\mu^{\otimes\N}$
and the one-coordinate orthonormality relation
\eqref{eq:r-orthonormality} give
$
\langle R_{\mathbf k}^{(\mu)},
R_{\mathbf l}^{(\mu)}\rangle_{L^2(\mathbb P_\mu)}
=
\delta_{\mathbf k\mathbf l}.
$ For every $F\in L^2(\mathbb P_\mu)$, the conditional
expectations given $X_1,\ldots,X_m$ converge to $F$
in $L^2$. Each finite-coordinate space has the
corresponding tensor-product basis. This proves
completeness and the expansion formula.

The identity $N_\mu=D_\mu^*D_\mu$ gives
$
\langle D_\mu r_a^{(\mu)},
D_\mu r_b^{(\mu)}\rangle_{L^2(\nu_\mu)}
=
\lambda_b(\mu)\delta_{ab}.
$
Applying this identity in each coordinate yields
\[
\left\langle
\mathbf D_\mu^0R_{\mathbf k}^{(\mu)},
\mathbf D_\mu^0R_{\mathbf l}^{(\mu)}
\right\rangle_{L^2(\mathbb P_\mu;\ell^2)}
=
\Lambda_{\mathbf k}(\mu)\delta_{\mathbf k\mathbf l}.
\]
Hence, for every finite expansion,
\[
\|\mathbf D_\mu^0F\|_{L^2(\mathbb P_\mu;\ell^2)}^2
=
\sum_{\mathbf k}
\Lambda_{\mathbf k}(\mu)|F_{\mathbf k}|^2.
\]

Suppose finite expansions $F_n$ converge to zero in
$L^2(\mathbb P_\mu)$ and
$\mathbf D_\mu^0F_n$ converges to $V$.
Every fixed coefficient of $F_n$ tends to zero.
The preceding orthogonality identity therefore makes
$V$ orthogonal to each
$\mathbf D_\mu^0R_{\mathbf k}^{(\mu)}$.
Since $V$ also belongs to their closed linear span,
we have $V=0$. This proves closability.

Completion in the graph norm gives the domain and
norm identity for $\mathbf D_\mu$.
The corresponding closed form is diagonal in the
basis $R_{\mathbf k}^{(\mu)}$, with diagonal entries
$\Lambda_{\mathbf k}(\mu)$. Its associated operator
is $\mathbf D_\mu^*\mathbf D_\mu$, which gives the
stated domain and action of $\mathbf N_\mu$. The proof is complete. \end{proof} 
 
We call $\mathbf N_\mu$ the {\bf product number operator}.
We now construct the main object of our study: the product spectral chaos by grouping
the product basis functions according to their
eigenvalues under $\mathbf N_\mu$.
Let 
$
S_\mu
=
\{\Lambda_{\mathbf k}(\mu):
\mathbf k\text{ has finite support}\}
$
be the set of distinct eigenvalues.
Every bounded interval contains only finitely many
elements of $S_\mu$. A sum at most $R$ contains at
most $\lfloor R/\lambda_1(\mu)\rfloor$ positive
summands, each taken from the finite set of
one-coordinate eigenvalues not exceeding $R$.

\begin{definition}
Grouping basis functions with the same eigenvalue gives
the orthogonal decomposition
\begin{equation}\label{eq:chaos}
L^2(\mathbb P_\mu)
=
\bigoplus_{\Lambda\in S_\mu}\mathcal C_\Lambda^{(\mu)},\quad
\mathcal C_\Lambda^{(\mu)}
=
\overline{\Span}
\left\{
R_{\mathbf k}^{(\mu)}:
\Lambda_{\mathbf k}(\mu)=\Lambda
\right\}
=
\ker(\mathbf N_\mu-\Lambda I),
\end{equation}
where the closure is taken in $L^2(\mathbb P_\mu)$.
We call \eqref{eq:chaos} the {\bf product spectral chaos}.
\end{definition}

The subspaces are indexed by eigenvalues $\Lambda$,
which need not be integers. This index differs from
the number of nonzero entries of $\mathbf k$.
Every positive eigenvalue has infinite multiplicity:
moving a fixed finite pattern of nonzero indices
to different coordinates produces infinitely many
orthogonal basis functions with the same eigenvalue.
The eigenspace for zero is
$
\mathcal C_0^{(\mu)}
=
\{c\,1:c\in\C\}.
$

At $\mu=1$, we have $\lambda_k(1)=k$.
The subspace with eigenvalue $n$ is the closed span
of the Hermite products of total degree
$
\sum_{j\geq1}k_j=n.
$
Thus \eqref{eq:chaos} is the usual Wiener chaos
decomposition on the Gaussian coordinate space.

For $\mu<1$, the construction is relative to the
chosen coordinates. In particular, $\nu_\mu$ is not
invariant under $x\mapsto-x$. Otherwise its positive
density relative to $\gamma$ would make
$e_0^{(\mu)}$ even. Comparing its differential
equation at $x$ and $-x$ would then give
$
c_{\mu,1}x^2e_0^{(\mu)}(x)=0$, $x>0$,
contrary to $c_{\mu,1}>0$ and $e_0^{(\mu)}>0$.
Thus invariance under arbitrary orthogonal changes
of Gaussian coordinates is not retained.
The decomposition for $\mu<1$ is defined through
the eigenspaces of $\mathbf N_\mu$, without a
representation by multiple Wiener integrals.

\subsection{The relation to Gaussian chaos}

For $0<\mu<1$, the measure $\nu_\mu(\mathrm dx)$ is non-Gaussian.
Indeed, if its density were Gaussian, then
\eqref{eq:ground} would imply
$
e_0^{(\mu)}(x)=C\exp(ax^2+bx)
$
for real $a,b$ and $C>0$.
Substitution into
$A_\mu e_0^{(\mu)}=\alpha_0(\mu)e_0^{(\mu)}$
on $x>0$ gives
\[
2a-4a^2=0,
\qquad
b-4ab=0,
\qquad
\alpha_0(\mu)=-2a-b^2.
\]
The choice $a=1/2$ is incompatible with
$e_0^{(\mu)}\in L^2(\gamma)$.
The remaining choice $a=0$ forces $b=0$ and
$\alpha_0(\mu)=0$, contradicting
Proposition~\ref{prop:limits}.

The one-coordinate measures $\nu_\mu(\mathrm dx)$ and $\gamma(\mathrm dx)$
are mutually absolutely continuous because their
density ratio is strictly positive. Their countable
products are mutually singular.

\begin{proposition}[Singularity of the countable product law]
\label{prop:singular}
For $0{\,<\,}\mu{\,<\,}1$,
\[
\nu_\mu^{\otimes\N}\perp\gamma^{\otimes\N}.
\]
Moreover, the sequence
\[
\prod_{j=1}^m e_0^{(\mu)}(X_j)
\]
does not converge in $L^2(\gamma^{\otimes\N})$,
where $X_j$ is the $j$th coordinate map on $\R^\N$.
\end{proposition}

\begin{proof}
Choose a Borel set $B$ such that
$\nu_\mu(B)\neq\gamma(B)$.
Under $\nu_\mu^{\otimes\N}$, the strong law gives
\[
\frac1m\sum_{j=1}^m\ind_B(X_j)
\longrightarrow\nu_\mu(B)
\qquad\text{almost surely}.
\]
Under $\gamma^{\otimes\N}$, the same averages converge
almost surely to $\gamma(B)$.
The limits differ, so the corresponding full-measure
sets are disjoint. This proves mutual singularity.

Independence and
$\|e_0^{(\mu)}\|_{L^2(\gamma)}=1$ also give
\[
\begin{aligned}
&\left\|
\prod_{j=1}^{m+1}e_0^{(\mu)}(X_j)
-
\prod_{j=1}^{m}e_0^{(\mu)}(X_j)
\right\|_{L^2(\gamma^{\otimes\N})}^2
=
2-2\int_{-\infty}^{\infty}
e_0^{(\mu)}(x)\,\gamma(\mathrm dx)
>0.
\end{aligned}
\]
The strict inequality follows from the Cauchy--Schwarz inequality and
since $e_0^{(\mu)}$ is positive, normalised and
nonconstant. The right-hand side is independent of
$m$, so the sequence is not Cauchy.
\end{proof}

The finite-coordinate unitary multiplication maps
therefore do not extend by taking an $L^2$ limit of
their multipliers on the Gaussian product space.
A separate measurable change of coordinates between
the probability spaces remains possible.  

\section{Rodrigues formulas on the spectral branches}
\label{sec:rodrigues}

We return to the one-coordinate eigenfunctions
$e_k^{(\mu,y)}$ of Theorem~\ref{thm:spectrum}.
Their branch representation \eqref{eq:ansatz} involves the
orders $\alpha_k(\mu)$ and $\beta_\mu(\alpha_k(\mu))$.
Recall the affine relation \eqref{eq:order}:
\[
\beta_\mu(\alpha)=\mu\alpha-\frac{1-\mu}{2}.
\]
We ask whether the fractional analog of the Rodrigues formula, applied to these factors, can yield another matched eigenfunction.
Such operators shift the order of a special function from $a$ to $a+s$.
We first establish this formula and its domain,
including negative initial orders, and then show why equal
nonzero changes of the two branch orders fail when $0<\mu<1$.

Throughout this section, $L^2=L^2(\R,\mathrm dx)$.
We use
\[
\widehat h(\xi)=\int_{\R}e^{-ix\xi}h(x)\,\mathrm dx,
\]
initially for integrable functions and then extended to
$L^2$ by Plancherel's theorem. For real $s$, define
\begin{equation}\label{eq:FROP}
\begin{aligned}
\mathcal D_\pm^s h
&=\mathcal F^{-1}\!\left[(\pm i\xi)^s\widehat h(\xi)\right],
\\
\Dom(\mathcal D_\pm^s)
&=\left\{h\in L^2:
(\pm i\xi)^s\widehat h(\xi)\in L^2\right\},
\end{aligned}
\end{equation}
where
$(\pm i\xi)^s
=|\xi|^s e^{\pm i\pi s\operatorname{sgn}(\xi)/2}$,
$\xi\neq0.$
These closed multipliers give the $L^2$ realisations of the
left- and right-sided Liouville operators, respectively.
For negative $s$, the singularity of the symbol near zero
makes the domain restriction essential.

For comparison with the classical definitions
\cite{Samko}, p. 94, the Liouville fractional integrals are
\[
(I_\pm^s\varphi)(x)
=\frac{1}{\Gamma(s)}
\int_0^\infty t^{s-1}\varphi(x\mp t)\,\mathrm dt,
\qquad s>0,
\]
whenever the integral converges. For $\varphi$ in the
Schwartz space $\mathcal S(\R)$ and $0<s<1$,
\begin{equation}\label{eq:FRIn}
\widehat{I_\pm^s\varphi}(\xi)
=(\pm i\xi)^{-s}\widehat\varphi(\xi)
\end{equation}
in tempered distributions, and in $L^2$ when the right-hand
side belongs to $L^2$. For $\varphi\in\mathcal S(\R)$,
$n\in\N$ and $n-1<s<n$, the corresponding derivatives are \cite{Samko}, p. 95
\begin{equation}\label{eq:FRDer}
(\mathcal D_\pm^s\varphi)(x)
=\frac{(\pm1)^n}{\Gamma(n-s)}
\frac{\mathrm d^n}{\mathrm dx^n}
\int_0^\infty t^{n-s-1}\varphi(x\mp t)\,\mathrm dt.
\end{equation}
At integer orders,
$\mathcal D_\pm^n=(\pm1)^n\mathrm d^n/\mathrm dx^n$,
and $\mathcal D_\pm^0=I$; for $s<0$,
$\mathcal D_\pm^s=I_\pm^{-s}$ on their common domain.

Using the Gaussian density $w_y$ defined by \eqref{GD}, we denote by $\mathcal R_y^s$ the next operator
\[
\mathcal R_y^sf
=y^sw_y^{-1}\mathcal D_-^s(w_yf),
\qquad w_yf\in\Dom(\mathcal D_-^s).
\]
The domain condition concerns $w_yf$ in  $L^2$.
At order one, this gives  
\[
\mathcal R_y^1f
=-yw_y^{-1}(w_yf)'=xf-yf'.
\]
The following proposition gives the action of
$\mathcal R_y^s$ on the functions $H_a$ defined in
\eqref{eq:H}, for real $s$.

\begin{proposition}[Rodrigues formula]\label{prop:rodrigues}
For real $a>-1/2$,
\begin{equation}\label{eq:Rodrigues}
w_yH_a(\cdot,y)=y^a\mathcal D_-^aw_y,
\end{equation}
where $\mathcal D_-^a$ is defined by \eqref{eq:FROP}.
Thus $\mathcal R_y^a1=H_a(\cdot,y)$, and for   $s$
with $a+s>-1/2$,
\[
\mathcal R_y^sH_a(\cdot,y)=H_{a+s}(\cdot,y).
\]
For $a,s,t\in\R$ satisfying
\[
a>-\tfrac12,\qquad
a+s>-\tfrac12,\qquad
a+s+t>-\tfrac12,
\]
the successive operations satisfy
\[
\mathcal R_y^t\bigl(\mathcal R_y^sH_a(\cdot,y)\bigr)
=\mathcal R_y^{s+t}H_a(\cdot,y)
=H_{a+s+t}(\cdot,y).
\]
Here the conditions apply to the initial order $a$,
the intermediate order $a+s$ and the final order $a+s+t$.
\end{proposition}

\begin{proof}
We first identify the domain of $\mathcal R_y^s$. Since
$\widehat w_y(\xi)=e^{-y\xi^2/2}$, the Gaussian controls
integrability at infinity, while
\[
\int_0^1 \xi^{2a}\,\mathrm d\xi<\infty
\quad\Longleftrightarrow\quad a>-\tfrac12.
\]
Thus $w_y\in\Dom(\mathcal D_-^a)$ exactly in this range;
the resulting Fourier transform also belongs to $L^1$.

We next prove the formula for $a{\,\in\,}(-1/2,0)$, where the
operator $\mathcal D_-^a$ is a fractional integral. By \eqref{eq:Dint},
\[
\frac{1}{\Gamma(-a)}
\int_0^\infty u^{-a-1}w_y(x+u)\,\mathrm du
=\frac{y^{-a/2}}{\sqrt{2\pi y}}
e^{-x^2/(4y)}D_a(x/\sqrt y).
\]
Multiplying the integration kernel by $e^{-\varepsilon u}$
makes it integrable and gives the Fourier multiplier
$(\varepsilon-i\xi)^a$.
As $\varepsilon\downarrow0$, dominated convergence in
$L^2$ on the Fourier side and monotone convergence in
the displayed integral identify its left-hand side with
$\mathcal D_-^aw_y$. Multiplication by $y^a$ and the
definition \eqref{eq:H} prove \eqref{eq:Rodrigues}
for $a{\,\in\,}(-1/2,0)$.

To extend the identity, fix $x$. Both sides of
\eqref{eq:Rodrigues} extend holomorphically to
$\operatorname{Re}a>-1/2$: for the Fourier integral,
domination on compact subsets remains valid after
inserting powers of $\log|\xi|$, and $D_a$ depends
holomorphically on its order. The identity theorem
therefore proves the formula throughout this half-plane.

Finally, \eqref{eq:Rodrigues} gives
\[
\widehat{w_yH_a(\cdot,y)}(\xi)
=y^a(-i\xi)^a e^{-y\xi^2/2}.
\]
Applying $\mathcal D_-^s$ multiplies this transform by
$(-i\xi)^s$. The identity
$(-i\xi)^s(-i\xi)^a=(-i\xi)^{a+s}$ yields
\[
y^s\mathcal D_-^s(w_yH_a(\cdot,y))
=w_yH_{a+s}(\cdot,y),
\]
provided $a,a+s>-1/2$.
Dividing by $w_y$ proves the change-of-order formula.
A second application gives the composition identity
when also $a+s+t>-1/2$. This completes the proof.
\end{proof}

We now apply the results of Proposition \ref{prop:rodrigues}  to the   special function factors in the two branches of  (\ref{eq:ansatz}), in particular $H_{\beta_\mu(\alpha_k(\mu))}(\cdot,\mu y)$ generated by 
$\beta_\mu(\alpha_k(\mu))$.
Every left order is covered: combining \eqref{eq:order}
with the nonnegativity of the eigenvalues in
\eqref{eq:bounds} gives
\[
\beta_\mu(\alpha_k(\mu))
=\mu\bigl(\alpha_k(\mu)+1/2\bigr)-1/2>-1/2.
\]
Negative orders are needed for the lowest eigenfunction.
For $0<\mu<1$, Proposition~\ref{prop:limits} gives
$0<\alpha_0(\mu)<1$. Using the positive ground state
from Theorem~\ref{thm:spectrum}, its branch representation
\eqref{eq:ansatz} and the origin values \eqref{eq:origin},
we obtain $(e_0^{(\mu)})'(0)>0$.
The matching conditions \eqref{eq:matrix} therefore require
$D_{\beta_\mu(\alpha_0(\mu))}'(0)<0$.
As $\beta_\mu(\alpha_0(\mu))\in(-1/2,1)$,
\eqref{eq:origin} then gives
$-1/2<\beta_\mu(\alpha_0(\mu))<0.$

To compare the changed orders, apply $\mathcal R_y^s$
to $H_{\alpha_k(\mu)}(x,y)$ on the right.
On the left, write $u=-x>0$ and apply
$\mathcal R_{\mu y}^s$ in the variable $u$ to
$H_{\beta_\mu(\alpha_k(\mu))}(u,\mu y)$,
leaving the multiplicative Gaussian weight
$\rho_{\mu,y}(-x)$ defined in \eqref{eq:rho} unchanged.
For $s>0$, these operations produce the orders
$\alpha_k(\mu)+s$ and $\beta_\mu(\alpha_k(\mu))+s$.
The affine relation \eqref{eq:order} instead requires
the left order
\[
\beta_\mu(\alpha_k(\mu)+s)
=\beta_\mu(\alpha_k(\mu))+\mu s.
\]
The difference between the produced and required orders is
\[
\bigl[\beta_\mu(\alpha_k(\mu))+s\bigr]
-\beta_\mu(\alpha_k(\mu)+s)
=(1-\mu)s.
\]
Thus equal nonzero changes cannot produce another
matched eigenfunction when $0<\mu<1$.
Multiplying the branches by nonzero constants cannot
remove this discrepancy,
since a nonzero function cannot solve the
parabolic-cylinder equation \eqref{eq:Dode} at two
distinct orders. The same argument covers
$s<0$ whenever both changed orders exceed $-1/2$.

Changes of $s$ and $\mu s$ preserve \eqref{eq:order},
but the new right order must still satisfy
$\Delta_\mu(\alpha_k(\mu)+s)=0$, with $\Delta_\mu$
defined in \eqref{eq:delta}, and the branch coefficients
must satisfy the matching conditions \eqref{eq:matrix}
at that new order.
At $\mu=1$, the spectrum is $\alpha_k(1)=k$
(Proposition~\ref{prop:limits}), so a noninteger
change leaves the spectrum. This conclusion concerns the
specified operations on the special-function factors,
with the Gaussian weight retained; it is not a formula
for differentiating the whole transported branch.
Fractional evolution of the fixed basis is defined
spectrally in Section~\ref{sec:fractional}.

\section{Spectral fractional dynamics and parameter limits}\label{sec:fractional}
\subsection{Spectral evolution and inverse clocks}

  We construct fractional evolution from the deformed number
operator $N_\mu$ (see Definition \ref{NumOp}). Fractional powers modify its spectral rates,
while a Caputo time derivative replaces exponential decay
by Mittag--Leffler relaxation.

 The ground-state transformation \eqref{eq:ground} of the
interface basis in Theorem~\ref{thm:spectrum} supplies
an orthonormal eigenbasis for $N_\mu$ in $L^2(\nu_\mu)$.
Standard
spectral calculus now gives fractional evolution by changing the scalar
factor attached to each eigenfunction. We distinguish the spatial power
$\sigma$, which changes spectral rates, from the Caputo order $\tau$,
which describes inverse-clock waiting intervals. The threshold
calculation uses $\sigma=1$ and $0<\tau<1$.

Fix $0<\mu\leq1$. Throughout this subsection, norms and inner products refer to $L^2(\nu_\mu)$. For $f\in L^2(\nu_\mu)$, let $f_k=\langle f,r_k^{(\mu)}\rangle$ denote its spectral coefficients, and let $P_0f=f_0\,1$ be its projection onto the constants. Recall that $r_0^{(\mu)}=1$, $N_\mu r_k^{(\mu)}=\lambda_k(\mu)r_k^{(\mu)}$, and $0=\lambda_0(\mu)<\lambda_1(\mu)<\cdots$.

 For $0<\sigma\leq1$, the spectral calculus
\cite[Eqs.~(12.10)--(12.11), p.~183]{Schilling}
defines
\[
(N_\mu/2)^\sigma f
=\sum_{k\geq0}(\lambda_k(\mu)/2)^\sigma f_k r_k^{(\mu)},
\]
on the domain consisting of those $f{\in}L^2(\nu_\mu)$ for which
$
\sum_{k\geq0}(\lambda_k(\mu)/2)^{2\sigma}|f_k|^2{<}\infty.
$
Thus the eigenfunctions remain unchanged, while their eigenvalues are replaced by fractional powers \cite{Balakrishnan,Schilling}. The semigroup $e^{-t(N_\mu/2)^\sigma}$ gives the mild solution of
\[
\partial_tu+(N_\mu/2)^\sigma u=0,
\qquad u(0)=f.
\]
For $f$ in the displayed operator domain, the solution is continuously differentiable in $L^2$ and continuous in the graph norm, including at $t=0$.

The probabilistic meaning follows from stable subordination. For $0<\sigma<1$, let $S_t^{(\sigma)}$ be the $\sigma$-stable subordinator normalised by $\E e^{-aS_t^{(\sigma)}}=e^{-ta^\sigma}$ for $a\geq0$. It is a nondecreasing L\'evy process. Applying its Laplace transform to each eigenfunction  gives
\[
e^{-t(N_\mu/2)^\sigma}f
=\int_0^\infty e^{-sN_\mu/2}f\,
\mathbb P(S_t^{(\sigma)}\in\mathrm ds).
\]
 Consequently, for $0<\sigma<1$, the fractional semigroup
$e^{-t(N_\mu/2)^\sigma}$ is obtained by time-changing the
diffusion with generator $(-N_\mu/2)$ by an independent
$\sigma$-stable subordinator $(S_t^{(\sigma)})_{t\ge0}$, normalized by
$
\mathbb E[e^{-\lambda S_t^{(\sigma)}}]=e^{-t\lambda^\sigma}$,
$\lambda,t\ge0
$
(see \cite[Remark~13.12, p.~211, and Eq.~(13.27), p.~216]{Schilling}).
At $\sigma=1$, the same formula gives the deterministic
clock $S_t^{(1)}=t$.

To introduce waiting intervals, take a separate order $0<\tau<1$   and the inverse stable clock $E_t^{(\tau)}=\inf\{s\geq0:S_s^{(\tau)}>t\}$. Intervals crossed by jumps of the subordinator become intervals on which its inverse is constant. The clock transform is expressed through the Mittag--Leffler function
\[
E_\tau(z)=\sum_{n=0}^\infty\frac{z^n}{\Gamma(\tau n+1)}:
\qquad \E e^{-aE_t^{(\tau)}}=E_\tau(-at^\tau),\quad a\geq0.
\]
For these identities and the inverse-clock interpretation,
see \cite[Sections~2 and~5.8]{Meerschaert}.

The corresponding time derivative is the Caputo derivative \cite{Kiryakova}
\[
{}^CD_t^\tau u(t)
=\frac1{\Gamma(1-\tau)}
\int_0^t(t-s)^{-\tau}u'(s)\dd s,
\]
whenever this Bochner integral exists. We consider the initial-value problem
\[
{}^CD_t^\tau u+(N_\mu/2)^\sigma u=0,\qquad u(0)=f.
\]
Its spectral solution is
\begin{equation}\label{eq:caputosolution}
u(t)=\sum_{k\geq0}f_k
E_\tau\!\left[-t^\tau(\lambda_k(\mu)/2)^\sigma\right]r_k^{(\mu)}.
\end{equation}
Pollard's complete-monotonicity theorem \cite{Pollard} gives $0\leq E_\tau(-a)\leq1$ for $a\geq0$. Parseval and dominated convergence therefore show that this series defines a bounded continuous $L^2$-valued function, with $u(0)=f$ and $\|u(t)\|_2\leq\|f\|_2$. Its Laplace transform is
\[
\int_0^\infty e^{-pt}u(t)\dd t
=p^{\tau-1}\bigl[p^\tau I+(N_\mu/2)^\sigma\bigr]^{-1}f,
\qquad p>0,
\]
which characterises the mild solution uniquely. Equivalently, one runs the process with semigroup $e^{-s(N_\mu/2)^\sigma}$ on an independent inverse clock. For bounded Borel data this is an ordinary expectation, which extends by contraction to $L^2(\nu_\mu)$. This is the fractional Cauchy construction of Baeumer and Meerschaert \cite[Theorem~3.1]{Baeumer}.

For domain data, the strong equation requires control near the initial time. The scaling identity $E_t^{(\tau)}\overset{d}=t^\tau E_1^{(\tau)}$ and the moment $\E E_1^{(\tau)}=1/\Gamma(1+\tau)$ allow differentiation of the scalar clock transform. Applying the resulting bound to the spectral coefficients gives
\[
\|u'(t)\|_2
\leq\frac{t^{\tau-1}}{\Gamma(\tau)}\|(N_\mu/2)^\sigma f\|_2,
\qquad f\in\Dom((N_\mu/2)^\sigma).
\]
The bound is integrable at zero. Hence $u$ is absolutely continuous on bounded time intervals, the operator image $(N_\mu/2)^\sigma u(t)$ is continuous, and termwise integration yields
\[
u(t)=f-\frac1{\Gamma(\tau)}
\int_0^t(t-s)^{\tau-1}(N_\mu/2)^\sigma u(s)\dd s,\qquad t>0.
\]
 At $\tau=1$, we use $E_1(z)=e^z$, recovering ordinary evolution. In every case the  constant eigenfunction is preserved.

The same spectral calculus gives bounded smoothing operators. 
For $\theta>0$, we define the Bessel-type and centred
Riesz-type potentials associated with $N_\mu$ by
the negative-power spectral calculus
\cite[Section~2, Eq.~(2.3)]{Stinga}:
$$
(I+N_\mu)^{-\theta}f
=\sum_{k\geq0}(1+\lambda_k(\mu))^{-\theta}f_kr_k^{(\mu)},\quad
N_\mu^{-\theta}(I-P_0)f
=\sum_{k\geq1}\lambda_k(\mu)^{-\theta}f_kr_k^{(\mu)}.
$$
The positive gap makes the second operator bounded; it is defined to
vanish on constants. Orders add under composition within either
family. Proposition~\ref{prop:limits} also gives the one-coordinate
Hilbert--Schmidt property for the Bessel potential when $\theta>1/2$.
The standard semigroup representation of positive powers
\cite{Balakrishnan,Schilling} and Theorem 1.1 from
\cite{Stinga} apply to $N_\mu/2$ for $0<\sigma<1$; in the extension formula
the  constant eigenfunction is retained separately from the  eigenfunctions corresponding to positive eigenvalues.

The evolution and potential formulas apply to the product operator
$\mathbf N_\mu$ by replacing $r_k^{(\mu)}$ and $\lambda_k(\mu)$ with
$R_{\mathbf k}^{(\mu)}$ and $\Lambda_{\mathbf k}(\mu)$. 
 For the countable product operator of
Theorem~\ref{thm:product}, every positive eigenvalue
has infinite multiplicity: placing a fixed finite
pattern of nonzero indices in disjoint sets of
coordinates gives infinitely many orthogonal
eigenfunctions with the same eigenvalue
\eqref{eq:prod}.
Consequently, its Bessel potentials are not
Hilbert--Schmidt.  For the collective calculation,
the essential fact is that a shared inverse clock acts on  the sum of coordinate eigenvalues for each product eigenfunction.

\subsection{Truncation and an interface coefficient identity}

Let $P_N$ project onto $r_0^{(\mu)},\ldots,r_N^{(\mu)}$. Parseval's identity and monotonicity of each multiplier give, for $0<\sigma,\tau\leq1$,
\begin{equation}\label{eq:tail}
\|(I-P_N)u(t)\|_2\leq E_\tau[-t^\tau(\lambda_{N+1}(\mu)/2)^\sigma]\|(I-P_N)f\|_2.
\end{equation}
The corresponding Bessel bound has multiplier $(1+\lambda_{N+1}(\mu))^{-\theta}$. For centred $f$, truncating $N_\mu^{-1}f$ at eigenfunction $N$ gives $L^2$ and derivative errors at most $\lambda_{N+1}(\mu)^{-1}\|f\|_2$ and $\lambda_{N+1}(\mu)^{-1/2}\|f\|_2$. These are analytical bounds; numerical certification requires a certified lower bound for the omitted gap, not just a root residual.

For a threshold observation, the coefficients can be calculated from values and derivatives at the interface. For $f=\ind_{(0,\infty)}$ put $p_\mu=\nu_\mu((0,\infty))$. Then $f_0=p_\mu$, and the Lagrange identity for $e_0^{(\mu)}$ and $e_k^{(\mu)}$ on the positive half-line gives
\begin{equation}\label{eq:indicator}
f_k=\frac{w_1(0)}{\lambda_k(\mu)}
\left[e_0^{(\mu)}(0)(e_k^{(\mu)})'(0)-(e_0^{(\mu)})'(0)e_k^{(\mu)}(0)\right],\quad k\geq1.
\end{equation}
Indeed $f_k=\int_0^\infty e_0^{(\mu)}e_k^{(\mu)}\dd\gamma$. Multiply the equation for $e_k^{(\mu)}$ by $e_0^{(\mu)}$, subtract the equation for the first eigenfunction multiplied by $e_k^{(\mu)}$, and integrate over the positive half-line. Integration by parts gives the displayed boundary flux. Polynomial growth on this half-line removes the term at infinity. No division by the interface value occurs, so eigenfunctions vanishing at the interface are included.

Since $\|f\|_2^2=p_\mu$, \eqref{eq:tail} becomes
\begin{equation}\label{eq:indicatortail}
\|u(t)-P_Nu(t)\|_2\leq E_\tau[-t^\tau(\lambda_{N+1}(\mu)/2)^\sigma]
\left(p_\mu-\sum_{k=0}^N|f_k|^2\right)^{1/2}.
\end{equation}
Every truncation preserves mean $p_\mu$. The full solution is a half-line probability under the appropriate clocks, but a finite sum need not lie in $[0,1]$: at $\mu{\,=\,}\sigma{\,=\,}\tau{\,=\,}1$, $N=1$, it is $1/2+e^{-t/2}x/\sqrt{2\pi}$. Thus \eqref{eq:indicatortail} is an $L^2$ error estimate, not a pointwise probability bound.

\subsection{Gaussian and radial Ornstein--Uhlenbeck endpoints}

The endpoints have a direct interpretation. As $\mu\uparrow1$, the
extra potential disappears and the equilibrium law becomes Gaussian.
As $\mu\downarrow0$, confinement suppresses the negative half-line;
the limiting eigenfunctions of the original operator have zero trace at the origin. The ground-state
transform of this Dirichlet limit produces the radial diffusion below.
 Because the probability measures vary with $\mu$,
we compare the evolution operators on the fixed space
$L^2(\gamma)$ using the unitary maps \eqref{eq:Gmu}. On $x>0$ define
\[
\nu_0(\mathrm dx)=2x^2\gamma(\mathrm dx),\qquad G_0f(x)=\sqrt2\,xf(x),\qquad
N_0=G_0^{-1}(A_{1,+}^D-I)G_0,
\]
where $A_{1,+}^D$ is the positive-half-line Dirichlet realisation of $(-\partial_{xx}+x\partial_x)$. The measure $\nu_0$ is a probability law and $G_0$ is unitary onto $L^2((0,\infty),\gamma)$. The expression for $N_0$ is 
$(-\partial_{xx}+(x-2/x)\partial_x)$; its domain is specified by the transform, without an extra boundary condition on this singular expression. Thus $(-N_0/2)$ is the radial Ornstein--Uhlenbeck generator in dimension three. Let $\iota_+$ denote zero extension and $\iota_+^*$ restriction.

\begin{theorem}[Endpoint laws and fractional operators]\label{thm:endpoints}
The laws $\nu_\mu$ converge in total variation to $\gamma$ as $\mu\uparrow1$ and to $\nu_0$, extended by zero on $x<0$, as $\mu\downarrow0$. For fixed $0<\sigma,\tau\leq1$ and $0<\varepsilon<T<\infty$, the operators
\[
G_\mu E_\tau[-t^\tau(N_\mu/2)^\sigma]G_\mu^{-1},\qquad E_\tau(z)=\sum_{n=0}^\infty\frac{z^n}{\Gamma(\tau n+1)},
\]
converge in norm on $L^2(\gamma)$, uniformly for $t\in[\varepsilon,T]$, to
\[
\begin{cases}
E_\tau[-t^\tau(A_1/2)^\sigma],&\mu\uparrow1,\\
\iota_+G_0E_\tau[-t^\tau(N_0/2)^\sigma]G_0^{-1}\iota_+^*,&\mu\downarrow0.
\end{cases}
\]
\end{theorem}
\begin{proof}
The argument has two parts: convergence of finitely many eigenprojections, then a uniform bound on the remaining multipliers. The compactness argument in Proposition~\ref{prop:limits} applies to each fixed eigenfunction. At $\mu\uparrow1$, testing its weak equation identifies every subsequential limit as a normalised Hermite eigenfunction of eigenvalue $k$. At $\mu\downarrow0$, the limit vanishes on $x<0$, has zero trace, and testing inside $x>0$ identifies the Dirichlet eigenvalue $2k+1$. Simplicity gives convergence of the rank-one orthogonal projections in operator norm, independently of the choice of signs. Positivity fixes the limits of the first eigenfunction:
\[
e_0^{(\mu)}(x)\longrightarrow1\quad(\mu\uparrow1),\qquad
e_0^{(\mu)}(x)\longrightarrow\sqrt2\,x\ind_{\{x>0\}}\quad(\mu\downarrow0)
\]
in $L^2(\gamma)$. The inequality $\|a^2-b^2\|_1\leq\|a-b\|_2\|a+b\|_2$ proves total-variation convergence.

As $\mu\downarrow0$, after choosing the signs consistently,
the eigenfunctions satisfy
\[
e_k^{(\mu)}(x)
\longrightarrow
\frac{\sqrt{2}\,\He_{2k+1}(x)}{\sqrt{(2k+1)!}}\,
\ind_{\{x>0\}}
\qquad\text{in }L^2(\gamma).
\]
On $(0,\infty)$, the limiting functions are the normalised
eigenfunctions of $A_{1,+}^D$, the Dirichlet realisation
of $A_1$ in $L^2((0,\infty),\gamma)$ with boundary
condition $f(0)=0$, introduced immediately before
Theorem~\ref{thm:endpoints}. They are extended by zero to
the negative half-line.
Since $\lambda_k(\mu)=\alpha_k(\mu)-\alpha_0(\mu)$,
Proposition~\ref{prop:limits} gives
\[
\lambda_k(\mu)\longrightarrow
\begin{cases}
k,&\mu\uparrow1,\\
2k,&\mu\downarrow0.
\end{cases}
\]

For the tails, \eqref{eq:bounds} and $0\leq\alpha_0(\mu)<1$ give $\lambda_k(\mu)\geq k-1$. For $M\geq1$ the norm of the sum with $k>M$ is at most
$E_\tau[-\varepsilon^\tau(M/2)^\sigma],$
uniformly in $\mu$ and $t\in[\varepsilon,T]$. It tends to zero as $M\to\infty$, and the limiting spectra obey the same bound. Truncation followed by $M\to\infty$ proves the result.
\end{proof}

For the endpoint $\mu\downarrow0$, the limiting evolution
first restricts the initial datum to $(0,\infty)$,
applies the half-line evolution, and then extends the
result by zero to $(-\infty,0)$.
It therefore annihilates data supported on the negative
half-line. At $t=0$, its value is the projection
$
\iota_+\iota_+^*f=\ind_{\{x>0\}}f,
$
whereas the evolution operator for every $\mu>0$
equals the identity at $t=0$.
Consequently, convergence to this endpoint cannot
include $t=0$ on the full space $L^2(\gamma)$.

\section{Concluding remarks}

The corrected interface system leads to a product spectral chaos after
a ground-state transform. The resulting number operator has
a closed gradient--divergence factorisation and reduces to the
classical Gaussian construction at $\mu=1$. Its one-coordinate dynamics have a radial three-dimensional Ornstein--Uhlenbeck limit. The endpoint convergence also holds for the fractional
evolution operators after the Hilbert spaces have been identified.

 The Rodrigues analysis answers a separate structural question. Its
domain includes every spectral branch, but equal nonzero increments
do not preserve the deformed order relation. A fractional stochastic
integral representation would therefore require a construction beyond
these branch shifts and will be studied elsewhere.

\section*{Statements and Declarations}
\paragraph{Funding and competing interests.} The work of the second author is an output of a research project HSE-BR-2026-39 implemented as part of the Basic Research Program at HSE University.
\paragraph{Data and code availability.}
No external dataset was used. The full application to collective threshold
correlations under shared and independent inverse stable clocks, including
the theorem and its proof, is available in the public repository
\cite{Code} at \url{https://github.com/colena/FWCII}.
The repository also provides Python code and numerical outputs for the
interface eigenfunctions, threshold coefficients, collective clock
comparison, Gaussian benchmark and truncation estimates.
\paragraph{Use of generative AI.} OpenAI ChatGPT was used for language revision, checking mathematical arguments and references, and developing the numerical reproduction code. The authors are responsible for verifying and approving the final mathematical content, computations, citations and conclusions.

\end{document}